\documentclass[reqno, 10pt]{amsart}
\usepackage{amssymb,mathrsfs}
\usepackage[usenames, dvipsnames]{color}
\usepackage{esint}
\usepackage{hyperref}
\usepackage{todonotes}
\usepackage{verbatim}

\usepackage{cite}
\usepackage[nobysame,non-sorted-cites]{amsrefs}
\def\MR#1{}

\theoremstyle{plain}
\newtheorem{theorem}{Theorem}[section]
\newtheorem{lemma}[theorem]{Lemma}
\newtheorem{corollary}[theorem]{Corollary}
\newtheorem{proposition}[theorem]{Proposition}

\theoremstyle{definition}

\theoremstyle{remark}
\newtheorem{remark}[theorem]{Remark}

\newcommand{\dv}{\operatorname{div}}
\newcommand{\osc}{\operatorname{osc}}

\newcommand{\dist}{\operatorname{dist}}
\newcommand{\diam}{\operatorname{diam}}

\newcommand{\sgn}{\operatorname{sgn}}

\numberwithin{equation}{section}

\newcommand{\bN}{\mathbb{N}}

\newcommand{\bR}{\mathbb{R}}

\newcommand{\bS}{\mathbb{S}}

\newcommand\cB{\mathcal{B}}

\newcommand\cD{\mathcal{D}}

\newcommand\cF{\mathcal{F}}

\newcommand\cK{\mathcal{K}}

\newcommand\sA{\mathscr{A}}

\makeatletter
\def\dashint{\operatorname%
{\,\,\text{\bf--}\kern-.98em\DOTSI\intop\ilimits@\!\!}}
\makeatother

\begin{document}
\title[The perfect conductivity problem]{Asymptotics of the solution to the perfect conductivity problem with $p$-Laplacian}
\author[H. Dong]{Hongjie Dong}

\author[Z. Yang]{Zhuolun Yang}

\author[H. Zhu]{Hanye Zhu}

\address[H. Dong]{Division of Applied Mathematics, Brown University, 182 George Street, Providence, RI 02912, USA}
\email{Hongjie\_Dong@brown.edu}

\address[Z. Yang]{Division of Applied Mathematics, Brown University, 182 George Street, Providence, RI 02912, USA}
\email{zhuolun\_yang@brown.edu}

\address[H. Zhu]{Division of Applied Mathematics, Brown University, 182 George Street, Providence, RI 02912, USA}
\email{hanye\_zhu@brown.edu}

\thanks{H. Dong was partially supported by the NSF under agreement DMS-2055244.}
\thanks{Z. Yang was partially supported by Simons Foundation Institute Grant Award ID 507536 and the AMS-Simons Travel Grant.}
\thanks{H. Zhu was partially supported by the NSF under agreement DMS-2055244.}


\begin{abstract}
We study the perfect conductivity problem with closely spaced perfect conductors embedded in a homogeneous matrix where the current-electric field relation is the power law $J=\sigma|E|^{p-2}E$. The gradient of solutions may be arbitrarily large as $\varepsilon$, the distance between inclusions, approaches to 0. To characterize this singular behavior of the gradient in the narrow region between two inclusions, we capture the leading order term of the gradient. This is the first gradient asymptotics result on the nonlinear perfect conductivity problem.
\end{abstract}

\maketitle

\section{Introduction and Main results}

Our study is instigated by the damage analysis in the fiber composite materials \cite{BASL}. Particularly, when fibers are closely packed and in high-contrast to the background matrix in terms of material properties, the electric field could be amplified by the composite micro-structure. In this article, we investigate the specific scenario in which the fiber inclusions are perfect conductors, and the background matrix follows the current-electric field relation described by the power law:
\begin{equation}\label{power_law}
J = \sigma |E|^{p-2} E, \quad p > 1,
\end{equation}
where $J$, $E$, and $\sigma$ represent current, electric field, and conductivity, respectively. This power law  has physical relevance across various materials, including dielectrics, plastic moulding, plasticity phenomena, viscous flows in glaciology, electro-rheological and thermo-rheological fluids. We refer to \cites{AntRod,BIK,GarKoh,Idiart,LevKoh,Ruzicka,Suquet} and the references therein.

Before stating our results, let us describe the mathematical setup: let $\Omega \subset \bR^n$ be a bounded domain with $C^{2}$ boundary, and let $\cD_{1}^{0}$ and $\cD_{2}^{0}$
be two $C^{2}$ open sets with $\diam(\cD_i^0)>c>0$ and $\dist(\cD_{1}^{0} \cup \cD_{2}^{0}, \partial \Omega) > c > 0$, touching at the origin with the inner normal direction of $\partial{\cD}_{1}^{0}$ being the positive $x_{n}$-axis. We write the variable $x$ as $(x',x_{n})$, where $x' \in \bR^{n-1}$. For $\varepsilon > 0$, translating $\cD_{1}^{0}$ and $\cD_{2}^{0}$ by ${\varepsilon}/{2}$ along the $x_{n}$-axis, we obtain
$$\cD_{1}^{\varepsilon}:=\cD_{1}^{0}+(0',{\varepsilon}/{2})\quad\mbox{and}\quad\,\cD_{2}^{\varepsilon}
:=\cD_{2}^{0}-(0',{\varepsilon}/{2}).$$
We denote $\widetilde{\Omega}^\varepsilon := \Omega \setminus \overline{(\cD_1^\varepsilon \cup \cD_2^\varepsilon)}$.

The perfect conductivity problem incorporating the power law \eqref{power_law} can be modeled by the following $p$-Laplace equation with $p >1$:
\begin{equation}\label{equinfty}
\left\{
\begin{aligned}
-\dv (|D u_\varepsilon|^{p-2} D u_\varepsilon) &=0  &&\mbox{in }\widetilde{\Omega}^\varepsilon,\\
u_\varepsilon &= U_i^\varepsilon &&\mbox{on}~\overline{\cD_{i}^\varepsilon},~i=1,2,\\
\int_{\partial \cD_i^\varepsilon} |D u_\varepsilon|^{p-2} D u_\varepsilon \cdot \nu &=0,  && i=1,2,\\
 u_\varepsilon &= \varphi  &&\mbox{on } \partial \Omega,
\end{aligned}
\right.
\end{equation}
where $\varphi\in{C}^{2}(\partial\Omega)$ is given, $\nu = (\nu_1, \ldots, \nu_n)$ denotes the outer normal vector on $\partial{\cD}_{1}^\varepsilon \cup \partial{\cD}^\varepsilon_{2}$ (pointing away from  ${\cD}_{1}^\varepsilon \cup {\cD}^\varepsilon_{2}$), and $U_1^\varepsilon$, $U_2^\varepsilon$ are two constants determined by \eqref{equinfty}$_3$. Here and throughout the paper, we adopt the notation
$$Du_\varepsilon\cdot \nu(x):= \lim_{t\to 0_+}\frac{u_\varepsilon(x+t\,\nu(x))-u_\varepsilon(x)}{t} \quad\mbox{and} \quad Du_\varepsilon(x):= [Du_\varepsilon\cdot \nu(x)]\nu(x)$$
for $x\in \partial \cD_i^\varepsilon$, $i=1,2$.

The solution $u_\varepsilon \in W^{1,p}
(\Omega)$ can be viewed as the unique function which has the minimal energy in appropriate function space: $I_p[u] = \min_{v \in \sA^\varepsilon}I_p[v]$, where
\begin{equation}\label{minimizer}
\begin{aligned}
I_p[v] &:= \int_{\Omega} |\nabla v|^p, \quad v \in \sA^\varepsilon,\\
\sA^\varepsilon &:= \{ v \in W^{1,p}(\Omega): \nabla v \equiv 0 ~~\mbox{in}~~\cD_1^\varepsilon \cup \cD_2^\varepsilon, v = \varphi ~~\mbox{on}~~\partial\Omega \}.
\end{aligned}
\end{equation}
 We refer the reader to the Appendix of \cite{BLY1} for the derivation of \eqref{equinfty} and its equivalence with \eqref{minimizer}. Although this derivation specifically addresses the case when $p=2$, the argument can be readily applied to $p>1$ with slight modifications.

The perfect conductivity problem \eqref{equinfty} with $p=2$ has undergone thorough studies. It was proved by Ammari et al. in \cite{AKL} and \cite{AKLLL} that, when $\cD_{1}$ and $\cD_{2}$ are disks of comparable radii in $\mathbb{R}^{2}$, the blow-up rate of the gradient of the solution is $\varepsilon^{-1/2}$ as $\varepsilon$ goes to zero; Yun in \cites{Y1,Y2} generalized the above mentioned result for two strictly convex inclusions in $\mathbb{R}^{2}$. These gradient estimates in dimension $n = 2$ were localized and extended to higher dimensions by Bao, Li, and Yin in \cite{BLY1}:
\begin{equation*}
\| \nabla u_\varepsilon \|_{L^\infty(\widetilde{\Omega}^\varepsilon)} \le
\left\{
\begin{aligned}
&C\varepsilon^{-1/2} \|\varphi\|_{C^2(\partial \Omega)} &&\mbox{when}~n=2,\\
&C|\varepsilon \ln \varepsilon|^{-1} \|\varphi\|_{C^2(\partial \Omega)} &&\mbox{when}~n=3,\\
&C\varepsilon^{-1} \|\varphi\|_{C^2(\partial \Omega)} &&\mbox{when}~n\ge 4.
\end{aligned}
\right.
\end{equation*}
These bounds were shown to be optimal in the paper and they are independent of the shape of inclusions, as long as the inclusions are relatively strictly convex. Moreover, numerous studies have been conducted into characterizing the asymptotic behavior of $\nabla u$, which are significant in practical applications. For further works on the linear perfect conductivity problem, see e.g. \cites{ACKLY,ADY,BT1,BT2,CY,DL,DZ,Gor,KLY1,KLY2,KL,L,LLY,LWX,LimYun} and the references therein.

The study on the nonlinear perfect conductivity problem \eqref{equinfty} is less comprehensive. The only results were given by Gorb and Novikov \cite{GorNov} and Ciraolo and Sciammetta \cite{CirSci}. They proved that for $n \ge 2$,
$$
\| \nabla u_\varepsilon\|_{L^\infty(\widetilde \Omega^\varepsilon)} \le
\left\{
\begin{aligned}
&C\varepsilon^{-\frac{n-1}{2(p-1)}} && \mbox{when}~p > \frac{n+1}{2},\\
&C\varepsilon^{-1}|\ln \varepsilon|^{\frac{1}{1-p}} && \mbox{when}~p = \frac{n+1}{2},\\
&C\varepsilon^{-1} && \mbox{when}~1<p < \frac{n+1}{2}.
\end{aligned}
\right.
$$
These bounds were shown to be optimal in their respective papers. In this paper, we give a more precise characterization of the gradient by capturing its leading order term in the asymptotics expansion.

It is noteworthy that, for the linear case in dimension two, solutions to the perfect conductivity problem and the insulated conductivity problem, representing the two extremes of conductivity, are harmonic conjugate to each other as shown in \cite{AKL}. Therefore, the behavior of their gradients is essentially identical due to the Cauchy-Riemann equation. The authors of this paper in \cite{DYZ23} studied the insulated conductivity problem with $p$-Laplacian, and identified the optimal blow-up exponent in dimension two. It turns out the gradient behaves significantly different from that of the solution to \eqref{equinfty} in dimension two. This showcases an intriguing feature of the nonlinear conductivity problem. For more results on the linear insulated conductivity problem, we refer to \cites{BLY2, DLY, DLY2, LY, LY2, We, Y3}.

To study the asymptotic behavior of $u_\varepsilon$, the solution to \eqref{equinfty}, it is important to study the limiting problem \eqref{minimizer} with $\varepsilon = 0$. We will show that the minimizing problem \eqref{minimizer} with $\varepsilon = 0$ is equivalent to
\begin{equation}\label{touching_equation_1}
\left\{
\begin{aligned}
-\dv (|D u_0|^{p-2} D u_0) &=0  &&\mbox{in }\widetilde{\Omega}^0,\\
u_0 &= U_0 &&\mbox{on}~\overline{\cD_{1}^0}\cup\overline{\cD_{2}^0},\\
\int_{\partial \cD_1^0 \cup \partial \cD_2^0} |D u_0|^{p-2} D u_0 \cdot \nu &=0,\\
 u_0 &= \varphi  &&\mbox{on } \partial \Omega
\end{aligned}
\right.
\end{equation}
for a constant $U_0$ when $p \ge (n+1)/2$, and is equivalent to
\begin{equation}\label{touching_equation_2}
\left\{
\begin{aligned}
-\dv (|D u_0|^{p-2} D u_0) &=0  &&\mbox{in }\widetilde{\Omega}^0,\\
u_0 &= U_i &&\mbox{on}~\overline{\cD_{i}^0}\setminus\{0\},~i=1,2,\\
\int_{\partial \cD_i^0} |D u_0|^{p-2} D u_0 \cdot \nu &=0, &&i=1,2,\\
 u_0 &= \varphi  &&\mbox{on } \partial \Omega
\end{aligned}
\right.
\end{equation}
for constants $U_1$ and $U_2$ when $p < (n+1)/2$. We would like to clarify a misunderstanding in the papers \cites{CirSci2, CirSci}. In \cites{CirSci2, CirSci}, the authors implicitly claimed that the minimizing problem \eqref{minimizer} with $\varepsilon = 0$ is universally equivalent to \eqref{touching_equation_1}, which is not the case. We will justify this in Theorem \ref{thm_A1}. We emphasize that while the minimizer $u_0$ of \eqref{minimizer} with $\varepsilon = 0$ always takes the same value in $\overline{\cD_{1}^0}$ and $\overline{\cD_{2}^0}$ when $p \ge (n+1)/2$, it may take different values when $1<p < (n+1)/2$. On the other hand, the flux along $\partial \cD_1^0$, denoted by
\begin{equation}\label{flux}
\cF:= \int_{\partial \cD_1^0} |D u_0|^{p-2} D u_0 \cdot \nu,
\end{equation}
might not be zero when $p \ge (n+1)/2$, but it must be zero when $p < (n+1)/2$.

By the regularity of $\cD_1^\varepsilon$ and $\cD_2^\varepsilon$, we can assume that near the origin,  the part of $\partial \cD_1^\varepsilon$ and $\partial \cD_2^\varepsilon$, denoted by $\Gamma_+^\varepsilon$ and $\Gamma_-^\varepsilon$, are respectively the graphs of two $C^2$ functions in terms of $x'$. That is,
\begin{align*}
\Gamma_+^\varepsilon = \left\{ x_n = \frac{\varepsilon}{2}+h_1(x'),~|x'|<1\right\},\quad \Gamma_-^\varepsilon = \left\{ x_n = -\frac{\varepsilon}{2}+h_2(x'),~|x'|<1\right\},
\end{align*}
where $h_1$ and $h_2$ are relatively convex $C^{2}$ functions satisfying
\begin{equation}\label{fg_0}
h_1(0')=h_2(0')=0,\quad\nabla_{x'}h_1(0')=\nabla_{x'}h_2(0')=0,
\end{equation}
\begin{equation}\label{fg_1}
c_1 |x'|^2 \le h_1(x')-h_2(x')\quad\mbox{for}~~0<|x'|<1,
\end{equation}
and
\begin{equation}\label{def:c_2}
    \|h_1\|_{C^{2}}\leq c_2, \quad\|h_2\|_{C^{2}}\leq c_2,
\end{equation}
with some positive constants $c_1,c_2$. For $x_0 \in \widetilde\Omega^\varepsilon$, $0 < r\leq 1$, we denote
\begin{align*}
\Omega_{x_0,r}^\varepsilon:=\left\{(x',x_{n})\in \widetilde{\Omega}^\varepsilon~\big|~-\frac{\varepsilon}{2}+h_2(x')<x_{n}<\frac{\varepsilon}{2}+h_1(x'),~|x' - x_0'|<r\right\},
\end{align*}
and $\Omega_{r}^\varepsilon:=\Omega_{0,r}^\varepsilon$. We also denote
$$\Gamma_{+,r}^\varepsilon := \Gamma_+^\varepsilon \cap \overline{\Omega}_r^\varepsilon, \quad \Gamma_{-,r}^\varepsilon := \Gamma_-^\varepsilon \cap \overline{\Omega}_r^\varepsilon.$$
We use $B_r(x_0)$ to denote the open ball of radius $r$ centered at $x_0$ and we set
$$
B_r:=B_r(0),\quad
\Omega_r^\varepsilon(x_0):=\widetilde\Omega^\varepsilon\cap B_r(x_0).
$$
Throughout this paper, we denote
\begin{equation}\label{def:delta}
    \underline{\delta}(x):= \varepsilon + |x'|^2
\end{equation}
and
\begin{equation}\label{def_delta}
\delta(x):= \varepsilon + h_1(x') - h_2(x').
\end{equation}
By \eqref{fg_0}--\eqref{def:c_2}, it can be easily seen that
$$\min\{c_1,1\}\underline{\delta}(x)\leq \delta(x) \leq \max\{c_2,1\}\underline{\delta}(x), \quad \text{for} \quad x\in \Omega_1.$$

We denote
\begin{equation}\label{Theta}
\Theta(\varepsilon):= \left\{
\begin{aligned}
& \varepsilon^\frac{2p-n-1}{2(p-1)}, && p > \frac{n+1}{2},\\
& |\ln \varepsilon|^{-\frac{1}{p-1}},  && p = \frac{n+1}{2},\\
&1, && 1<p < \frac{n+1}{2},
\end{aligned}
\right.
\end{equation}
and
\begin{equation}\label{def:K}
K = \left\{
\begin{aligned}
&\frac{\det\big( D_{x'}^2 (h_1 - h_2)(0') \big)^{\frac{1}{2}} \Gamma(p-1)}{(2\pi)^{\frac{n-1}{2}} \Gamma\big(p- \frac{n+1}{2} \big)}, &&\mbox{when}~~p > \frac{n+1}{2},\\
&\frac{\det\big( D_{x'}^2 (h_1 - h_2)(0') \big)^{\frac{1}{2}} \Gamma\big(\frac{n-1}{2} \big)}{(2\pi)^{\frac{n-1}{2}}}, &&\mbox{when}~~p = \frac{n+1}{2},
\end{aligned}
\right.
\end{equation}
where $\Gamma(z):= \int_0^\infty t^{z-1} e^{-t}$ is the Gamma function defined for $z > 0$.

Our main result is the following asymptotic expansion of $Du_\varepsilon(x)$ for sufficiently small $\varepsilon$ and $x'$.
\begin{theorem}\label{thm:expansion}
Let $h_1$, $h_2$ be  $C^2$ functions satisfying \eqref{fg_0}-\eqref{def:c_2}, $p>1$, $n \ge 2$, $u_\varepsilon \in W^{1,p}(\Omega)$ be the solution of \eqref{equinfty}, $u_0$ be the minimizer of \eqref{minimizer} with $\varepsilon = 0$, $\cF$ be given in \eqref{flux}, $U_1$, $U_2$ be the constants in \eqref{touching_equation_2}$_2$, $\delta(x)$ be defined in \eqref{def_delta}, $\Theta(\varepsilon)$ be given in \eqref{Theta}, and $K$ be defined in \eqref{def:K}. Then there exist constants $\beta\in(0,1)$ depending only on $n$ and $p$, and $C_1, C_2>0$ depending only on $n$, $p$, $c_1$, and $c_2$, such that the following holds:
\begin{itemize}
\itemindent=-13pt
    \item[(i)] If $p\ge (n+1)/2$, for $\varepsilon \in(0,1)$ and $x\in \Omega_{1/4}^\varepsilon$, we have
\begin{equation}\label{eq:expansion1}
Du_\varepsilon(x)=\big(0',\delta(x)^{-1}\Theta(\varepsilon)(\sgn(\cF)(K|\cF|)^{1/(p-1)}+f_0(\varepsilon))\big)+\mathbf{f}_1(x,\varepsilon),
\end{equation}
where $f_0:\mathbb{R}\to \mathbb{R}$ is a function of $\varepsilon$ and $\mathbf{f}_1:\mathbb{R}^n\times \mathbb{R}\to \mathbb{R}^n$ is a function of $x$ and $\varepsilon$, such that
\begin{equation*}
    \lim_{\varepsilon\to0}f_0(\varepsilon)=0,
\end{equation*}
\begin{equation}\label{eq:thm-f1}
    |\mathbf{f}_1(x,\varepsilon)|\le C_1 \Big( {\delta}(x)^{\beta/2-1}\Theta(\varepsilon)(|K\cF|^{1/(p-1)}+|f_0(\varepsilon)|) + \|\varphi\|_{L^\infty(\partial\Omega)}e^{-\frac{C_2}{\sqrt\varepsilon + |x'|}} \Big).
\end{equation}
\item[(ii)] If $1<p< (n+1)/2$, for $\varepsilon \in(0,1)$ and $x\in \Omega_{1/4}^\varepsilon$, we have
\begin{equation*}
    D u_\varepsilon(x)=\big(0',\delta(x)^{-1}(U_1-U_2+g_0(\varepsilon))\big)+\mathbf{g}_1(x,\varepsilon),
\end{equation*}
where
$g_0:\mathbb{R}\to \mathbb{R}$ is a function of $\varepsilon$ and $\mathbf{g}_1:\mathbb{R}^n\times \mathbb{R}\to \mathbb{R}^n$ is a function of $x$ and $\varepsilon$, such that
\begin{equation*}
    \lim_{\varepsilon\to 0}g_0(\varepsilon)=0,
\end{equation*}
\begin{equation}\label{eq:thm-g1}
    |\mathbf{g}_1(x,\varepsilon)|\le C_1 \Big( {\delta}(x)^{\beta/2-1}(|U_1-U_2|+|g_0(\varepsilon)|)+ \|\varphi\|_{L^\infty(\partial\Omega)}e^{-\frac{C_2}{\sqrt\varepsilon + |x'|}} \Big).
\end{equation}
\end{itemize}
\end{theorem}

 As a consequence of the asymptotic expansion in Theorem \ref{thm:expansion}, we provide a pointwise upper bound of $Du_\varepsilon$.
\begin{remark}
    \label{thm_upperbound}
Under the hypotheses of Theorem \ref{thm:expansion}, there exist constants $C_1, C_2>0$ depending only on $n$, $p$, $c_1$, and $c_2$, such that  for sufficiently small $\varepsilon>0$, and any $x\in \Omega_{1/4}^\varepsilon$, we have
\begin{equation}\label{grad_upperbound}
|D u_\varepsilon(x)| \le C_1 \|\varphi\|_{L^\infty(\partial \Omega)} \Big(\frac{\Theta(\varepsilon)}{\varepsilon + |x'|^2} + e^{-\frac{C_2}{\sqrt{\varepsilon} + |x'|}}\Big).
\end{equation}
In fact,
when $1<p < (n+1)/2$, \eqref{grad_upperbound} follows directly from Proposition \ref{prop_gradient_upperbound_U1-U2}.
When $p \ge (n+1)/2$, since $u_0 = U_0$ on $\overline{\cD_{1}^0}\cup\overline{\cD_{2}^0}$, we have $|\cF| \le C\|\varphi\|^{p-1}_{L^\infty(\partial \Omega)}$ and thus \eqref{grad_upperbound} follows from \eqref{eq:expansion1}. Indeed, one can see the boundedness of $Du_0$ on $\Gamma_{+, 1/2}^0$ by Lemma \ref{lemma_exp_decay}, and on $\partial D_1^0 \setminus \Gamma_{+, 1/2}^0$ by classical gradient estimates (see e.g. \cites{Lie}).
\end{remark}

 Another direct consequence of Theorem \ref{thm:expansion} is the following pointwise positive lower bound of $D_n u_\varepsilon$ near the origin, provided the coefficient of the leading order term in the asymptotic expansion is positive.

\begin{remark}
    \label{thm_lowerbound}
Under the hypotheses of Theorem \ref{thm:expansion}, if either
\begin{equation}\label{condition_1}
p \ge \frac{n+1}{2}~~\mbox{and}~~\cF > 0
\end{equation}
or
\begin{equation}\label{condition_2}
1 < p < \frac{n+1}{2}~~\mbox{and}~~U_1 > U_2
\end{equation}
holds, then there exist constants  $\kappa_1, \kappa_2\in(0,1/4)$, $\gamma>0$ depending only on $n$, $p$, $c_1$, $c_2$, $\|\varphi\|_{L^\infty(\partial\Omega)}$, $\cF$ (when $p \ge (n+1)/2$), and $U_1 - U_2$ (when $p < (n+1)/2$), such that for sufficiently small $\varepsilon>0$, and any
$x\in \Omega_{1/4}^\varepsilon$ satisfying
$$
\left\{\begin{aligned}
    &|x'|\leq |\ln \varepsilon|^{-\gamma},  &&\text{when} \quad p>\frac{n+1}{2},\\
    &|x'|\leq \kappa_1(\ln|\ln \varepsilon|)^{-1},  &&\text{when} \quad p=\frac{n+1}{2},\\
    &|x'|\leq \kappa_2,  &&\text{when} \quad 1<p<\frac{n+1}{2},
\end{aligned}
\right.
$$
we have
\begin{equation}\label{grad_lowerbound}
\left\{\begin{aligned}
    &D_n u_\varepsilon (x) \ge\frac{1}{2}\delta(x)^{-1}\Theta(\varepsilon)(K\cF)^{1/(p-1)},  &&\text{when} \quad p\ge\frac{n+1}{2},\\
    &D_n u_\varepsilon (x) \ge \frac{1}{2} \delta(x)^{-1}(U_1-U_2),  &&\text{when} \quad 1<p<\frac{n+1}{2}.
\end{aligned}
\right.
\end{equation}
Indeed, when $p\ge (n+1)/2$, \eqref{grad_lowerbound} follows directly from Theorem \ref{thm:expansion} by setting
    $$
    \left\{
    \begin{aligned}
    &|f_0(\varepsilon)|\leq \frac{1}{6}(K\cF)^{1/(p-1)},\\
    & C_1 \delta(x)^{\beta/2}\leq \frac{1}{7},\\
    & C_1\|\varphi\|_{L^\infty(\partial\Omega)}e^{-\frac{C_2}{\sqrt\varepsilon + |x'|}}\leq \frac{1}{6}\delta(x)^{-1}\Theta(\varepsilon)(K\cF)^{1/(p-1)},
    \end{aligned}
    \right.
    $$
and the case when $p\in(1,(n+1)/2)$ follows similarly.
\end{remark}

Next, we provide a concrete example whose coefficient of the leading order term in the asymptotic expansion is positive.
\begin{proposition}\label{thm_example}
Let $\Omega = B_5\subset \bR^n$, $\cD_1 = B_2(0',2)$, $\cD_2 = B_2(0', -2)$, $\varphi = x_n$, and $u_0$ be the minimizer of \eqref{minimizer} with $\varepsilon = 0$. Then either \eqref{condition_1} or \eqref{condition_2} is satisfied.
\end{proposition}

Our proof of the main result Theorem \ref{thm:expansion} relies on the following $C^\beta$ bound of the gradient, which may be of independent interest.
\begin{proposition}
\label{thm-1/2}
Let $h_1$, $h_2$ be  $C^2$ functions satisfying \eqref{fg_0}-\eqref{def:c_2}, $p>1$, $n \ge 2$, $\varepsilon \in(0,1)$, and $u_\varepsilon \in W^{1,p}({\Omega})$ be a solution of \eqref{equinfty}. Then there exist constants $\beta\in(0,1)$ depending only on $n$ and $p$, and $C>0$ depending only on $n$, $p$, $c_1$, and $c_2$, such that for any $x\in \Omega_{1/4}$ and $\underline{\delta}(x)=\varepsilon+|x'|^2$, it holds that
\begin{equation}\label{gradient-1/2}
[D u_\varepsilon]_{C^\beta (\Omega_{x,\sqrt{\underline{\delta}(x)}/4})} \le  C \underline{\delta}(x)^{-\beta/2}\|D u_\varepsilon\|_{L^\infty(\Omega_{x,\sqrt{\underline{\delta}(x)}/2})}.
\end{equation}
\end{proposition}

\begin{remark}
    It can be seen from the proof in Section \ref{sec:mean} that Proposition \ref{thm-1/2} holds as long as $u_\varepsilon\in W^{1,p}({\Omega}_1^\varepsilon)$ is a solution of
\begin{equation*}
\left\{
\begin{aligned}
-\dv (|D u_\varepsilon|^{p-2} D u_\varepsilon) &=0  &&\mbox{in }\Omega_1^\varepsilon,\\
u_\varepsilon &= U_1^\varepsilon &&\mbox{on}~\Gamma_{+}^\varepsilon,\\
u_\varepsilon &= U_2^\varepsilon &&\mbox{on}~\Gamma_{-}^\varepsilon,
\end{aligned}
\right.
\end{equation*}
for some arbitrary constants $U_1^\varepsilon$ and $U_2^\varepsilon$. Moreover, the same estimates also hold for any solution $u_\varepsilon\in W^{1,p}({\Omega}_1^\varepsilon)$ of
\begin{equation*}
\left\{
\begin{aligned}
-\dv (|D u_\varepsilon|^{p-2} D u_\varepsilon) &=0  &&\mbox{in }\Omega_1^\varepsilon,\\
\frac{\partial u_\varepsilon}{\partial \nu} &=0 &&\mbox{on}~\Gamma_{\pm}^\varepsilon,\\
\end{aligned}
\right.
\end{equation*}
which might be useful for obtaining sharper blow-up estimates for the insulated conductivity problem with $p$-Laplacian (see \cite{DYZ23}).
\end{remark}

We briefly describe the steps of proving Theorem \ref{thm:expansion}. First, we establish a pointwise upper bound of the gradient in terms of $U_1^\varepsilon - U_2^\varepsilon$ for arbitrary given $U_1^\varepsilon$ and $U_2^\varepsilon$ (Proposition \ref{prop_gradient_upperbound_U1-U2}). Then we use mean oscillation estimates to prove a $C^{1,\beta}$ estimate (Proposition \ref{thm-1/2}). Note that Proposition \ref{thm-1/2} implies a power gain of order $\delta^{\beta/2}$ for the oscillation of the gradient in the $x_n$ direction. Because of this power gain and Proposition \ref{prop_gradient_upperbound_U1-U2}, we then derive an asymptotic expansion of $Du_\varepsilon$ in terms of $U_1^\varepsilon - U_2^\varepsilon$ (Proposition \ref{prop_gradient_lowerbound_U1-U2}). When $p \ge (n+1)/2$, $U_1^\varepsilon - U_2^\varepsilon$ will converge to $0$ as $\varepsilon \to 0$. In this case, we use the flux conditions to derive the convergence rate for $U_1^\varepsilon - U_2^\varepsilon$ (Theorem \ref{thm_U1-U2}). When $p < (n+1)/2$, $U_1^\varepsilon - U_2^\varepsilon$ will converges to $U_1 - U_2$ (Theorem \ref{thm_convergence}).  Finally, Theorem \ref{thm:expansion} follows from putting all the ingredients above together.

We would like to point out that weaker versions of Proposition \ref{prop_gradient_lowerbound_U1-U2} were proved in \cites{GorNov,CirSci}. They derived an asymptotic expansion of $Du_\varepsilon$ only on the upper and lower boundaries $\Gamma_{\pm}^\varepsilon$ by constructing suitable barrier functions and using comparison principle. However,  in our opinion, it appears that there is a gap in their proofs. It is nontrivial to see that the normal derivative is bounded in the case when $U_1^\varepsilon - U_2^\varepsilon$ is small as lack of control of the oscillation of the solution in the $x'$ direction (see \cite{CirSci}*{p.6174} and \cite{GorNov}*{p.740}). This gap can be filled by Proposition \ref{prop_gradient_upperbound_U1-U2} of this paper. We would like to remark that our argument in Proposition \ref{prop_gradient_lowerbound_U1-U2} is more robust in the sense that  the proof does not rely on the fundamental solution of the $p$-Laplace equation or the maximum principle. In fact, the only place this paper involves the maximum principle is Proposition \ref{prop_gradient_upperbound_U1-U2}. See also Remark \ref{remark-Ui}. If one can give an alternative proof of \eqref{gradient_upperbound_U1-U2-new} in Proposition \ref{prop_gradient_upperbound_U1-U2} without using the maximum principle, then our results can be extended to nonlinear systems of $p$-Laplace type.

The rest of the paper is organized as follows: In Section \ref{Sec_prelim}, we provide some preliminary estimates and results. In Section \ref{sec:mean}, we use mean oscillation estimates to prove Proposition \ref{thm-1/2}. A convergence rate of $U_1^\varepsilon - U_2^\varepsilon$ when $p\ge (n+1)/2$ is provided in Section \ref{sec:U1-U2}. Finally, the proofs of Theorem \ref{thm:expansion} and Proposition \ref{thm_example} are given in Section \ref{sec:proofs}.

\section{Preliminaries}\label{Sec_prelim}
In this section, we provide some preliminary results.

\begin{lemma}\label{lemma_exp_decay}
Let $h_1$, $h_2$ be  $C^2$ functions satisfying \eqref{fg_0}-\eqref{def:c_2}, $p>1$, $n \ge 2$, $\varepsilon \in[0,1)$, and $v \in W^{1,p}(\Omega_1^\varepsilon)$ be a solution of
\begin{equation}\label{p_laplace_narrow_0boundary}
\left\{
\begin{aligned}
-\dv (|D v|^{p-2} D v) &=0  &&\mbox{in }\Omega_1^\varepsilon,\\
v &= 0 &&\mbox{on}~\Gamma_{\pm}^\varepsilon.
\end{aligned}
\right.
\end{equation}
Then there exist constants $C_1, C_2>0$ depending only on $n$, $p$, $c_1$, and $c_2$, such that
\begin{equation}\label{eq:exp_decay}
|v(x)| + |Dv(x)| \le C_1 e^{-\frac{C_2}{\sqrt{\varepsilon} + |x'|}} \|v\|_{L^p(\Omega_1^\varepsilon\setminus \Omega_{1/2}^\varepsilon)} \quad \mbox{for}~~x \in \overline{\Omega_{1/2}^\varepsilon}.
\end{equation}
\end{lemma}

\begin{proof}
The proof of this lemma essentially follows that of \cite{BLLY}*{Theorem 1.1}, with some modification. For simplicity, we omit the superscript $\varepsilon$ in the proof.
Without loss of generality, we may assume $\varepsilon\in[0, 1/256)$ and $|x'|<1/16$ since otherwise \eqref{eq:exp_decay} follows from classical estimates for the $p$-Laplace equation (see e.g. \cite{Lie}). For any $0<t<s<1$, let $\eta = \eta(x')$ be a cutoff function such that $\eta = 1$ in $\Omega_t$, $\eta = 0$ in $\Omega_1 \setminus \Omega_s$, and $|D \eta| \le C (s-t)^{-1}$. Multiplying $v\eta^p$ on both sides of \eqref{p_laplace_narrow_0boundary} and integrating by parts, we have
$$
\int_{\Omega_1} |D v|^p \eta^p + p |D v|^{p-2} D v \cdot D\eta v \eta^{p-1} = 0.
$$
By Young's inequality,
$$
\int_{\Omega_t} |D v|^p \le \frac{C}{(s-t)^p} \int_{\Omega_s \setminus \Omega_t} |v|^p.
$$
Since $v = 0$ on $\Gamma_-$, by the Poincar\'e inequality in the $x_n$ direction, we have
$$
\int_{\Omega_s \setminus \Omega_t} |v|^p \le C(\varepsilon+s^2)^p \int_{\Omega_s \setminus \Omega_t} |D u|^p.
$$
Therefore,
\begin{equation} \label{iteration}
\int_{\Omega_t} |D v|^p \le C^* \left(\frac{\varepsilon+s^2}{s-t} \right)^p \int_{\Omega_s \setminus \Omega_t} |D v|^p.
\end{equation}
Let $t_0 = r \in (\sqrt{\varepsilon},1/2)$ and $t_j = (1-jr)r$ for $j \in \bN$ such that $j \le 1/r$. Taking $s = t_j, t = t_{j+1}$ in \eqref{iteration}, we have
$$
\int_{\Omega_{t_{j+1}}} |D v|^p \le 2^pC^* \int_{\Omega_{t_{j}} \setminus \Omega_{t_{j+1}}} |D v|^p.
$$
Adding both sides by $2^pC^*\int_{\Omega_{t_{j+1}}} |D v|^p$ and dividing both sides by $1+ 2^pC^*$, we have
$$
\int_{\Omega_{t_{j+1}}} |D v|^p \le \frac{2^pC^*}{1+2^pC^*} \int_{\Omega_{t_{j}}} |D v|^p.
$$
Let $k = \lfloor \frac{1}{2r} \rfloor$ and iterate the above inequality $k$ times. We have
\begin{equation} \label{grad_v_Lp}
\int_{\Omega_{r/2}}|D v|^p \le \left( \frac{2^pC^*}{1+2^pC^*} \right)^k \int_{\Omega_{r}}|D v|^p \le C\mu^{\frac{1}{r}} \int_{\Omega_{1}\setminus \Omega_{1/2}}|v|^p,
\end{equation}
where $\mu \in (0,1)$ and $C$ are constants depending only on $n$, $p$, $c_1$, and $c_2$. Now we take $r = 4(\sqrt{\varepsilon} + |x'|)$ and \eqref{eq:exp_decay} follows from classical estimates for the $p$-Laplace equation in $\Omega_{x', {\varepsilon} + |x'|^2}$ and \eqref{grad_v_Lp}.
\end{proof}

Next, we derive a pointwise upper bound of the gradient in terms of $U_1^\varepsilon - U_2^\varepsilon$.

\begin{proposition}\label{prop_gradient_upperbound_U1-U2}
Let $h_1$, $h_2$ be  $C^2$ functions satisfying \eqref{fg_0}-\eqref{def:c_2}, $p>1$, $n \ge 2$, $\varepsilon \in[0,1)$, $U_1^\varepsilon$, $U_2^\varepsilon$ be arbitrary constants, and $u_\varepsilon \in W^{1,p}(\Omega_1^\varepsilon)$ be a solution of
\begin{equation*}
\left\{
\begin{aligned}
-\dv (|D u_\varepsilon|^{p-2} D u_\varepsilon) &=0  &&\mbox{in }\Omega_1^\varepsilon,\\
u_\varepsilon &= U_1^\varepsilon &&\mbox{on}~\Gamma_{+}^\varepsilon,\\
u_\varepsilon &= U_2^\varepsilon &&\mbox{on}~\Gamma_{-}^\varepsilon.
\end{aligned}
\right.
\end{equation*}
Then there exist constants $C_1, C_2>0$ depending only on $n$, $p$, $c_1$, and $c_2$, such that for $x\in\Omega_{1/4}^\varepsilon$, it holds that
\begin{equation}\label{gradient_upperbound_U1-U2}
|Du_\varepsilon(x)| \le C_1 \Big(\frac{|U_1^\varepsilon - U_2^\varepsilon|}{\varepsilon + |x'|^2} + \|u_\varepsilon\|_{L^\infty(\Omega_1^\varepsilon \setminus \Omega_{1/2}^\varepsilon)} e^{-\frac{C_2}{\sqrt\varepsilon + |x'|}}\Big).
\end{equation}
Moreover, if $\varepsilon\in (0,1)$, $u_\varepsilon\in W^{1,p}(\Omega)$ is the solution to \eqref{equinfty} and $U_1^\varepsilon$, and $U_2^\varepsilon$ are the same constants in \eqref{equinfty}, we have
\begin{equation}\label{U1U2_upperbound}
\inf_{\partial \Omega} \varphi \le U_1^\varepsilon,\, U_2^\varepsilon \le \sup_{\partial \Omega} \varphi,
\end{equation}
and for $x\in\Omega_{1/4}^\varepsilon$,
\begin{equation}\label{gradient_upperbound_U1-U2-new}
|Du_\varepsilon(x)| \le C_1 \Big(\frac{|U_1^\varepsilon - U_2^\varepsilon|}{\varepsilon + |x'|^2} + \|\varphi\|_{L^\infty(\partial \Omega)} e^{-\frac{C_2}{\sqrt\varepsilon + |x'|}}\Big).
\end{equation}
\end{proposition}

\begin{proof}
We first give the proof of \eqref{gradient_upperbound_U1-U2}. Take a point $x_0 \in \Omega_{1/4}^\varepsilon$. In order to estimate the gradient at $x_0$, we first estimate the oscillation of $u_\varepsilon$ in $\Omega_{x_0, \underline{\delta}(x_0)/8}^\varepsilon$, where $\underline{\delta}(x_0)= \varepsilon + |x_0'|^2$. Without loss of generality, we may assume that $U_1^\varepsilon \ge U_2^\varepsilon$. Let $v$ be the solution to
$$
\left\{
\begin{aligned}
-\dv (|D v|^{p-2} D v) &= 0  &&\mbox{in }\Omega_1^\varepsilon,\\
v &= 0 &&\mbox{on}~\Gamma_{\pm}^\varepsilon,\\
v &= u_\varepsilon - U_1^\varepsilon &&\mbox{on}~~\partial\Omega_1^\varepsilon\cap\{x\in \mathbb{R}^n: |x'| = 1 \}.
\end{aligned}
\right.
$$
By Lemma \ref{lemma_exp_decay},
$$
|v(x)| \le C_1 e^{-\frac{C_2}{\sqrt{\varepsilon} + |x'|}} \|u_{\varepsilon}\|_{L^\infty( \Omega_1^\varepsilon \setminus \Omega_{1/2}^\varepsilon)} \quad \mbox{for}~~x \in \Omega_{1/2}^\varepsilon.
$$
Since $v \ge u_\varepsilon - U_1^\varepsilon$ on $\partial \Omega_1^\varepsilon$, by the comparison principle, we have
$$
u_\varepsilon(x) - U_1^\varepsilon \le v(x) \quad\mbox{in}~~\Omega_1^\varepsilon.
$$
Similarly, let $w$ be the solution to
$$
\left\{
\begin{aligned}
-\dv (|D w|^{p-2} D w) &= 0  &&\mbox{in }\Omega_1^\varepsilon,\\
w &= 0 &&\mbox{on}~\Gamma_{\pm}^\varepsilon,\\
w &= u_\varepsilon - U_2^\varepsilon &&\mbox{on}~~\partial\Omega_1^\varepsilon\cap\{x\in \mathbb{R}^n: |x'| = 1 \}.
\end{aligned}
\right.
$$
We have
$$
|w(x)| \le C_1 e^{-\frac{C_2}{\sqrt{\varepsilon} + |x'|}} \|u_{\varepsilon}\|_{L^\infty(\Omega_1^\varepsilon \setminus \Omega_{1/2}^\varepsilon)} \quad \mbox{for}~~x \in \Omega_{1/2}^\varepsilon,
$$
and
$$
u_\varepsilon(x) - U_2^\varepsilon \ge w(x) \quad\mbox{in}~~\Omega_1^\varepsilon.
$$
Therefore,
$$
\underset{\Omega_{x_0,\underline{\delta}(x_0)/8}^\varepsilon}{\osc} u_\varepsilon\le |U_1^\varepsilon - U_2^\varepsilon| + C_1\|u_\varepsilon\|_{L^\infty(\Omega_1^\varepsilon \setminus \Omega_{1/2}^\varepsilon)} e^{-\frac{C_2}{\sqrt\varepsilon + |x_0'|}}.
$$
Then the gradient estimate \eqref{gradient_upperbound_U1-U2} follows from classical boundary and interior estimates for the $p$-Laplace equation (see e.g. \cite{Lie}).

Next we prove \eqref{U1U2_upperbound}. Indeed, if {$U_i^\varepsilon = \max\{U_1^\varepsilon,U_2^\varepsilon \} > \sup_{\partial \Omega} \varphi$}, by the maximum principle and the Hopf lemma (see \cite{Vaz}*{Theorem 5}), $Du_\varepsilon \cdot \nu>0$ on $\partial \cD_i^\varepsilon$, which violates \eqref{equinfty}$_3$.

Finally, \eqref{gradient_upperbound_U1-U2-new} follows directly from \eqref{gradient_upperbound_U1-U2}, \eqref{U1U2_upperbound}, and the maximum principle.
\end{proof}
\begin{remark}\label{remark-Ui}
    For systems of $p$-Laplace type, instead of \eqref{U1U2_upperbound}, one can still show that
$$|U_i^\varepsilon| \leq C \|\varphi\|_{C^1(\partial \Omega)}, \quad i=1,2$$
holds for some $\varepsilon$-independent constant $C$, by using classical boundary and interior gradient estimates away from the neck region $\Omega_{1/2}^\varepsilon$. However, it is not clear to us whether \eqref{gradient_upperbound_U1-U2-new} (or a weaker version of it) is still true.
\end{remark}

In the following, we justify the equivalence between the minimizing problem \eqref{minimizer} with $\varepsilon = 0$ and the equations \eqref{touching_equation_1}-\eqref{touching_equation_2}.

\begin{theorem}\label{thm_A1}
$u_0$ is the minimizer of \eqref{minimizer} with $\varepsilon = 0$ if and only if $u_0\in W^{1,p}(\Omega)$ satisfies \eqref{touching_equation_1} when $p \ge (n+1)/2$ and satisfies \eqref{touching_equation_2} when $p < (n+1)/2$.
\end{theorem}

\begin{proof}
First, we prove that \eqref{touching_equation_1} has at most one solution $u \in W^{1,p}(\Omega)$. The same conclusion applies to \eqref{touching_equation_2}. Let $u_1, u_2\in W^{1,p}(\Omega)$ be two solutions of \eqref{touching_equation_1}. Multiplying the equation by $u_1 - u_2$ and integrating by parts, we have for $j = 1,2$,
\begin{align*}
0 &= \int_{\widetilde\Omega^0} |D u_j|^{p-2}D u_j \cdot D(u_1 - u_2)  \, dx- \int_{\partial \Omega} |D u_j|^{p-2}D u_j \cdot \nu (u_1 - u_2) \, dS\\
&\quad - \sum_{i=1}^2 \int_{\partial D_i^0} |D u_j|^{p-2}D u_j \cdot \nu (u_1 - u_2) \, dS\\
&= \int_{\widetilde\Omega^0} |D u_j|^{p-2}D u_j \cdot D(u_1 - u_2)  \, dx.
\end{align*}
Therefore,
\begin{align*}
0 &= \int_{\widetilde\Omega^0} \Big( |D u_1|^{p-2}D u_1 - |D u_2|^{p-2}D u_2 \Big) \cdot D(u_1 - u_2)  \, dx \\
&\ge \frac{\min\{1, p-1\}}{2^{p-2}}\int_{\widetilde\Omega^0} (|D u_1| + |D u_2|)^{p-2}|D u_1 - D u_2|^2\, dx.
\end{align*}
This implies $u_1 \equiv u_2$. It is straightforward to see that the minimizer of \eqref{minimizer} with $\varepsilon = 0$ is unique, due to the convexity of $I_p$ and $\sA^0$. It suffices to show that the minimizer $u_0$ satisfies \eqref{touching_equation_1} when $p \ge (n+1)/2$ and satisfies \eqref{touching_equation_2} when $p < (n+1)/2$. We show this by taking different test function $v \in \sA^0$ in the equation
\begin{equation}\label{critical_point}
0 = \left.\frac{d}{dt}  I_p[u_0 + t v] \right|_{t = 0}.
\end{equation}
First we take $v \in C_c^\infty (\widetilde\Omega^0)$. Then \eqref{critical_point} reads as
$$
0 = \int_{\widetilde\Omega^0} |D u_0|^{p-2} D u_0 \cdot D v \, dx.
$$
This implies
$$
-\dv (|D u_0|^{p-2} D u_0) =0  \quad\mbox{in }\widetilde{\Omega}^0.
$$
Next, we take $v \in C_c^\infty (\Omega)$ such that $v = 1$ in $\overline{\cD_1^0 \cup \cD_2^0}$. From \eqref{critical_point} and integration by parts, we have
\begin{align*}
0 =& \int_{\widetilde\Omega^0} |D u_0|^{p-2} D u_0 \cdot D v \, dx\\
=& - \int_{\widetilde\Omega^0} \dv (|D u_0|^{p-2} D u_0) v \, dx + \sum_{i=1}^2 \int_{\partial \cD_i^0} |D u_0|^{p-2} D u_0 \cdot \nu \, dS\\
=& \sum_{i=1}^2 \int_{\partial \cD_i^0} |D u_0|^{p-2} D u_0 \cdot \nu \, dS.
\end{align*}
For the case when $p \ge (n+1)/2$, it remains to show that $u_0$ equals to the same constant on $\overline{\cD_1^0}$ and $\overline{\cD_2^0}$. Assume that $u_0 = U_1$ in $\overline{\cD_1^0}$ and $u_0 = U_2$ in $\overline{\cD_2^0}$ with $U_1 \neq U_2$. Then  by the fundamental theorem of calculus,
$$
U_1 - U_2 = \int_{h_2(x')}^{h_1(x')} D_n u_0 (x) \, dx_n.
$$
Taking  the absolute value and raising to the power of $p$ on the both sides, by H\"older's inequality, \eqref{fg_0}, and \eqref{def:c_2}, we have
$$
|U_1 - U_2|^p \le C |x'|^{2(p-1)} \int_{h_2(x')}^{h_1(x')} |D_n u_0 (x)|^p \, dx_n.
$$
This implies
$$
\int_{|x'|<1/2} \frac{|U_1 - U_2|^p}{|x'|^{2(p-1)}} \, dx' \le C \int_{\Omega_{1/2}^0} |D u_0|^p \, dx.
$$
The left-hand-side diverges since $p \ge (n+1)/2$ and $U_1 \neq U_2$, which leads to a contradiction. Therefore, $U_1 = U_2$.

For the case when $p < (n+1)/2$, we need to show the flux on each of $\partial \cD_i^0$ vanishes. We will only show the flux vanishes on $\partial \cD_1^0$, as a similar argument applies to $\partial \cD_2^0$. Let $v$ be a function compactly supported in $\Omega$ such that $v = 1$ in $\overline{\cD_1^0}$, $v = 0$ in $\overline{\cD_2^0}$,
$$
v(x',x_n) = \left(\frac{2x_n - (h_1(x') + h_2(x'))}{h_1(x') - h_2(x')}\right)_+, \quad x \in \Omega^0_{1/2},
$$
and $v$ is smooth in $\widetilde\Omega^0 \setminus \Omega^0_{1/2}$. Then $v \in \sA^0$. Indeed, we only need to verify
\begin{align*}
\int_{\Omega^0_{1/2}} |\nabla v|^p \, dx \le C \int_{|x'|< 1/2} \int_{\frac{h_1(x') + h_2(x')}{2}}^{h_1(x')} \frac{1}{|x'|^{2p}} \, dx_n dx' \le C \int_{0}^{1/2} r^{n-2p} \le C
\end{align*}
since $n - 2p > -1$. Taking this $v$ in \eqref{critical_point} and integrating by parts, we have
$$
\int_{\partial \cD_1^0} |D u_0|^{p-2} D u_0 \cdot \nu \, dS = 0.
$$
The theorem is proved.
\end{proof}

Next, we show that $u_\varepsilon$ converges to $u_0$ in the following sense.

\begin{theorem}\label{thm_convergence}
Let $u_\varepsilon \in W^{1,p}(\Omega)$ be the solution of \eqref{equinfty}, and $u_0\in W^{1,p}(\Omega)$ be the minimizer of \eqref{minimizer} with $\varepsilon = 0$. Then as $\varepsilon \to 0$, $u_\varepsilon \rightharpoonup u_0$ weakly in $W^{1,p}(\Omega)$, and $u_\varepsilon \to u_0$ strongly in $C^{1,\beta}(K)$ for some $\beta > 0$ and any
$$K \subset \subset \Omega \setminus \Big( \underset{0 < \varepsilon \le \varepsilon_0}{\cup}(\cD_1^\varepsilon \cup \cD_2^\varepsilon) \cup \{0\} \Big) \quad \mbox{with }\varepsilon_0>0.$$ As a consequence,  as $\varepsilon \to 0$, $U_1^\varepsilon \to U_1$ and $U_2^\varepsilon \to U_2$ when $p < (n+1)/2$, and $U_1^\varepsilon, U_2^\varepsilon \to U_0$ when $p \ge (n+1)/2$, where $U_0,U_1,U_2$ are the constants in \eqref{touching_equation_1} and \eqref{touching_equation_2}.
\end{theorem}

\begin{proof}
First we take an arbitrary function $w \in W^{1,p}(\Omega)$ such that $w = \varphi$ on $\partial \Omega$ and $Dw = 0$ in $\cB$, where $\cB \subset \Omega$ is an open set containing $\overline{\cD_1^\varepsilon \cup \cD_2^\varepsilon}$. Therefore,
$$
\|u_\varepsilon\|_{W^{1,p}(\Omega)} \le \|w\|_{W^{1,p}(\Omega)}.
$$
In particular, $\|u_\varepsilon\|_{W^{1,p}(\Omega)}$ is bounded uniformly in $\varepsilon$. Then there exists a subsequence $\{\varepsilon_j\}_{j \in \bN}$ and a function $u_* \in W^{1,p}(\Omega)$, such that $u_{\varepsilon_j} \rightharpoonup u_*$ weakly in $W^{1,p}(\Omega)$, and $u_{\varepsilon_j} \to u_*$ strongly in $L^p(\Omega)$, as $j \to \infty$. From \eqref{U1U2_upperbound}, we know that $U_1^{\varepsilon_j}$ and $U_2^{\varepsilon_j}$, the values of $u_{\varepsilon_j}$ in $\cD_1^{\varepsilon_j}$ and $\cD_2^{\varepsilon_j}$, are uniformly bounded. Then there exists a subsequence of $\{\varepsilon_j\}_{j \in \bN}$, still denoted by $\{\varepsilon_j\}_{j \in \bN}$, so that $U_i^{\varepsilon_j} \to U_i$ for some constants $U_i$, $i=1,2$. Therefore, for any $x \in \cD_i^0$, we have
$$
u_*(x) = \lim_{j \to \infty} U_i^{\varepsilon_j} = U_i, \quad i=1,2.
$$
On the other hand, for any $k,l\in \mathbb{Z}_{+}$, we denote
\begin{equation*}
    \cK_{k,l}:= \Omega\setminus \Big( \underset{0 < \varepsilon \le 1/k}{\cup}(\cD_1^\varepsilon \cup \cD_2^\varepsilon) \cup B_{1/l}(0)\Big).
\end{equation*}
By the classical $C^{1,\alpha}$ estimate, when $\varepsilon_j \le 1/k$, we have
$$
\| u_{\varepsilon_j} \|_{C^{1,\alpha}(\overline{\cK}_{k,l})} \le C(k,l).
$$
This implies that there is a subsequence that converges in $C^{1,\beta}(\overline{\cK}_{k,l})$ for any $\beta < \alpha$. We can apply the Cantor diagonal argument to select a subsequence, still denoted by $\{\varepsilon_j\}_{j \in \bN}$, such that
\begin{equation}\label{C_1_beta_converge}
u_{\varepsilon_j} \to u_{**} ~~\mbox{in}~~ C^{1,\beta}(K)  \quad \mbox{as}~~j \to \infty
\end{equation}
for any $K  \subset \subset \Omega \setminus \Big( \underset{0 < \varepsilon \le \varepsilon_0}{\cup}(\cD_1^\varepsilon \cup \cD_2^\varepsilon) \cup \{0\} \Big)$ with $\varepsilon_0>0$ and some $C^{1,\beta}$ function $u_{**}$.
Therefore, $u_{**}$ is a weak solution to the $p$-Laplace equation in
$\widetilde\Omega^0$.
Since $u_\varepsilon \to u_*$ strongly in $L^p(\Omega)$, we have $u_* \equiv u_{**}$.  It remains to show that $u_* \equiv u_0$, which implies the convergence of $\{u_\varepsilon\}$.

When $p \ge (n+1)/2$, by the same argument as in the proof of Theorem \ref{thm_A1}, we know that $U_1 = U_2$. It remains to prove that
\begin{equation}\label{eq:sum-flux-zero}
\int_{\partial \cD_1^0 \cup \partial \cD_2^0} |D u_*|^{p-2} D u_{*} \cdot \nu =0.
\end{equation}
Let $\Omega'$ be an open set such that
$ \cD_1^0 \cup \cD_2^0 \subset\subset \Omega'\subset\subset \Omega$.
Since $u_{\varepsilon_j}$ is the solution to \eqref{equinfty}, when $j$ is sufficiently large, by integration by parts in $\Omega'\setminus \overline{\cD_1^{\varepsilon_j} \cup \cD_2^{\varepsilon_j}}$, we have
$$\int_{\partial \Omega'}  |D u_{\varepsilon_j}|^{p-2} D u_{\varepsilon_j} \cdot \nu =0.$$
Therefore, by \eqref{C_1_beta_converge}, we also have
\begin{equation}\label{flux_K_delta}
    \int_{\partial \Omega'}  |D u_{*}|^{p-2} D u_* \cdot \nu =0.
\end{equation}
Since $u_*$ is a weak solution to the $p$-Laplace equation in
$\widetilde\Omega^0$, by integration by parts again in $\Omega'\setminus \overline{\cD_1^{0} \cup \cD_2^{0}}$, \eqref{flux_K_delta} directly implies \eqref{eq:sum-flux-zero}.

When $p < (n+1)/2$, it remains to prove that
$$
\int_{\partial \cD_i^0} |D u_*|^{p-2} D u_* \cdot \nu =0, \quad i=1,2.
$$
We will prove it only for $i=1$. Fix a small $s \in (0,1/2)$, we take a smooth surface $\eta$ so that $\cD_1^\varepsilon$ is surrounded by $\Gamma_{-,s}^0 \cup \eta$. See Figure \ref{zeroflux}.
\begin{figure}[h]
       \centering
       \includegraphics[scale=0.3]{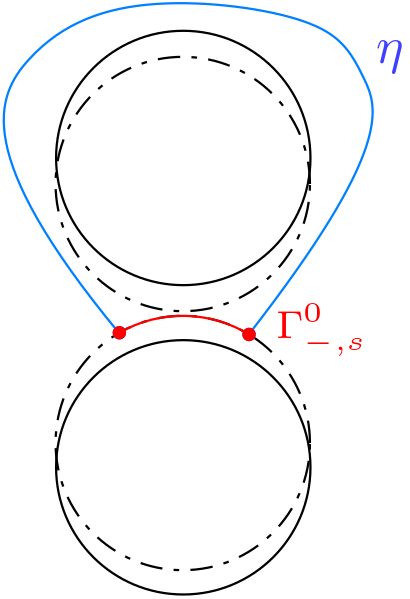}
       \caption{}\label{zeroflux}
\end{figure}

Since $\int_{\partial \cD_1^{\varepsilon_j}} |D u_{\varepsilon_j}|^{p-2} D u_{\varepsilon_j} \cdot \nu =0$, by integration by parts, we have
$$
- \int_{\Gamma_{-,s}^0} |D u_{\varepsilon_j}|^{p-2} D u_{\varepsilon_j} \cdot \nu + \int_{\eta} |D u_{\varepsilon_j}|^{p-2} D u_{\varepsilon_j} \cdot \nu =0.
$$
Note that the minus sign appears because $\nu$ on $\Gamma_{-,s}^0$ is pointing upwards, while $\nu$ on $\eta$ is pointing away from $\cD_1^\varepsilon$. By \eqref{gradient_upperbound_U1-U2-new}, we have $|D u_{\varepsilon_j}(x)| \le C(\varepsilon_j + |x'|^2)^{-1}$. Therefore,
\begin{align*}
\left| \int_{\Gamma_{-,s}^0} |D u_{\varepsilon_j}|^{p-2} D u_{\varepsilon_j} \cdot \nu \right| \le C \int_{|x'|< s} \frac{1}{|x'|^{2p-2}} \, dx' \le C s^{n-2p + 1},
\end{align*}
where we used $2p -2 < n-1$. By \eqref{C_1_beta_converge}, we know that
$$
\int_{\eta} |D u_{\varepsilon_j}|^{p-2} D u_{\varepsilon_j} \cdot \nu \to \int_{\eta} |D u_{*}|^{p-2} D u_{*} \cdot \nu \quad \mbox{as}~~j \to \infty.
$$
Therefore,
$$
 \left|\int_{\eta} |D u_{*}|^{p-2} D u_{*} \cdot \nu \right| \le C s^{n-2p + 1}.
$$
Similarly by \eqref{gradient_upperbound_U1-U2}, we have $|D u_{*}(x)| \le C|x'|^{-2}$ and \begin{align*}
\left| \int_{\Gamma_{-,s}^0} |D u_{*}|^{p-2} D u_{*} \cdot \nu \right| \le C s^{n-2p + 1}.
\end{align*}
By integration by parts, we have
\begin{align*}
\left|  \int_{\cD_1^0} |D u_{*}|^{p-2} D u_{*} \cdot \nu \right| &= \left| \int_{\eta} |D u_{*}|^{p-2} D u_{*} \cdot \nu - \int_{\Gamma_{-,s}^0} |D u_{*}|^{p-2} D u_{*} \cdot \nu \right|\\
&\le C s^{n-2p + 1}.
\end{align*}
Sending $s \to 0$  and using $p< (n+1)/2$, we have
$$
\int_{\cD_1^0} |D u_{*}|^{p-2} D u_{*} \cdot \nu = 0.
$$

Finally, by the uniqueness of solution to \eqref{touching_equation_1} and \eqref{touching_equation_2}, we can conclude that $u_* \equiv u_0$, and the full sequence $u_\varepsilon$ converges to $u_0$ in the corresponding topology.
\end{proof}

\section{Mean oscillation estimates}\label{sec:mean}
In this section, we give the proof of Proposition \ref{thm-1/2} using mean oscillation estimates.
Throughout this section, unless otherwise specified, we use $C$ to denote positive constants depending only on $n$, $p$, $c_1$, and $c_2$, which could differ from line to line. Here $c_1$ and $c_2$ are the same constants in \eqref{fg_1} and \eqref{def:c_2}, respectively. For simplicity, we denote $u:=u_\varepsilon$ and we omit the superscript $\varepsilon$ throughout this section when there is no confusion.

First, we fix a point $\Bar{x}\in\Omega_{1/2}$
and derive some mean oscillation estimates of $Du$ on a ball intersecting $\Omega_1$, namely $\Omega_r(\Bar{x})$, for different radii $r$.

\subsection{Mean oscillation estimates for small \texorpdfstring{$r$}{r}}

We recall a classical interior mean oscillation estimate when $B_{r}(\Bar{x})\subset{\Omega}_1$. Estimates of this type, with different exponents involved, were developed in  \cites{lieberman1991natural, dibenedetto1993higher, DuzMin10}.
\begin{lemma}\label{lem:mean1}
Let $u\in W^{1,p}({\Omega}_1)$ be a solution to \eqref{equinfty}. There exist constants $C>1$ and $\alpha\in(0,1)$ depending only on $n$ and $p$, such that $u\in C^{1,\alpha}({\Omega}_1)$ and for every $B_{r}(\Bar{x})\subset{\Omega}_1$ and $\rho\in (0,r]$, we have
\begin{align*}
    \phi(\Bar{x},\rho)\leq C\left(\frac{\rho}{r}\right)^{\alpha } \phi(\Bar{x},r),
\end{align*}
where we denote
\begin{align}\label{def:phi}
    \phi(\Bar{x},r)=\left(\fint_{\Omega_r(\Bar{x})}|D u-(D u)_{\Omega_r(\Bar{x})}|^p \right)^{\frac{1}{p}}.
\end{align}
\end{lemma}
\subsection{Mean oscillation estimates for intermediate \texorpdfstring{$r$}{r}}
Next, we consider the case when $B_{r}(\Bar{x})$ intersects with only one of $\Gamma_+$ and $\Gamma_-$.
In this case, we derive mean oscillation estimates around  any $\hat{x}\in (\Gamma_+\cup\Gamma_-)\cap\{x\in\mathbb{R}^n: |x'|\leq 3/4\}$.


Without loss of generality, let $\hat{x}\in \Gamma_-\cap\{x\in \mathbb{R}^n: |x'|\leq 3/4\}$. Then by \eqref{fg_1} and \eqref{def:c_2}, there exists a constant $c=c(n,c_1,c_2)\in(0,\min\{c_1/4,1/4\})$, such that $B(\hat{x},r)\cap \Gamma_+=\emptyset$ for any $r\in(0,c\underline{\delta}(\hat{x}))$. Here we recall \eqref{def:delta}. We first choose a coordinate  system $y=(y',y_n)$ such that $y(\hat{x})=0$, the direction of axis $y_n$ is the normal vector at $\hat{x}\in\Gamma_-$ pointing upwards.
Note that the coordinate $y$ is a rotation (plus a transition) of the coordinate $x$, namely $y=T(x-\hat{x})$ for some rotation matrix $T\in\mathbb{R}^{n\times n}$, which maps the normal vector of $\Gamma_-$ at $\hat{x}$ pointing upward, namely $\nu=(-Dh_2(\hat{x}'),1)/\sqrt{1+|Dh_2(\hat{x}')|^2}$ in the $x$-coordinate, to the unit vector $\mathbf{e}_{y_n}=(0,\ldots,0,1)$ in $y$-coordinate. Therefore, by \eqref{fg_0} and \eqref{def:c_2}, there exists a constant $C=C(n,c_2)>0$, such that
\begin{align}\label{rotation}
    |T-I_n|\leq C |\hat{x}'| \quad \text{and} \quad |T^{-1}-I_n|\leq C |\hat{x}'|.
\end{align}
Thus there exists a constant $c_3=c_3(n,c_1,c_2)\in(0,\min\{c_1/8,1/8\})$ such that
$\Omega_{R_0}(\hat{x})=\{y\in B_{R_0}:y_n>\chi(y')\}$, where $R_0=c_3\underline{\delta}(\hat{x})\in(0,1/4)$ and
$\chi:\{y'\in\mathbb{R}^{n-1}:|y'|<R_0\}\rightarrow\mathbb{R}$ is a $C^2$ function in the $y$-coordinate system such that
\begin{equation}\label{def:chi}
\chi(0')=0,\quad D_{y'}\chi(0^{\prime})=0,\quad \|\chi\|_{C^2}\leq C \|h_2\|_{C^2},
\end{equation}
for some constant $C=C(n)>0$.
Then we let
$$
z=\Lambda(y) =(y',\,y_n-\chi(y')).
$$
Since $\Gamma_-$ is $C^2$, by \eqref{def:chi} there exist constants
\begin{equation}\label{def:c_4}
    c_4=c_4(n,c_1,c_2)\in(0,\min\{c_1/8,1/8\}),
\end{equation}
$C=C(n,c_1,c_2)>0$, and $R_1=c_4\underline{\delta}(\hat{x})$ such that
\begin{equation*}
   |D_{y'}\chi(y')|\leq C\,|y'|\leq 1/2 \quad \text{if} \quad |y'|\leq 2R_1,
\end{equation*}
\begin{equation*}
    \Omega_{r/2}(\hat{x})\subset \Lambda^{-1}(B^+_r)\subset \Omega_{2r}(\hat{x}) \quad \forall\, r\in(0,2R_1],
\end{equation*}
and thus
\begin{equation}\label{co-diff}
   |D\Lambda(y)-I_n|\leq C\,|y'|\leq 1/2 \quad \text{if} \quad |y'|\leq 2R_1.
\end{equation}
Therefore, there exist positive constants $c(n)$ and $c'(n)$ depending only on $n$, such that for any $\hat{x}\in (\Gamma_+\cup\Gamma_-)\cap \{x\in\mathbb{R}^n: |x'|\leq 3/4\}$ and $0<r\leq c_4\underline{\delta}(\hat{x})$,
\begin{equation}\label{volume1}
    c(n) r^n\leq  |\Omega_r(\hat{x})| \leq c'(n)r^n.
\end{equation}
Note that
\begin{align*}
    \text{det}(D\Lambda)\equiv 1.
\end{align*}
Then $u_1(z):=u(\Lambda^{-1}(z))$ satisfies the following equation with constant Dirichlet boundary condition
\begin{equation}\label{eq:u1}
    \left\{
\begin{aligned}
     -\dv_z\left(|A^T D_z u_1|^{p-2}AA^T  D_z u_1\right)=\;&0 \quad  &&\text{in} \,\, B^+_{R_1},\\
   u_1=\;&U_2^{\varepsilon} \quad &&\text{on}\,\, B_{R_1}\cap\partial \mathbb{R}^n_+,
\end{aligned}
\right.
\end{equation}
where we denote
$$
A:={A}(z):=(a_{ij}(z)):=D{\Lambda}(\Lambda^{-1}(z)) .
$$
Next we extend the equation to the whole ball $B_{R_1}$.
We take the even extension of $a_{nn}$ and $a_{ij},i,j = 1,2,\ldots, n-1,$ with respect to $z_n = 0$, and take the odd extension of $a_{in}$ and $a_{ni},i = 1,2,\ldots, n-1,$ with respect to $z_n = 0$. Then we reflect $u_1$ with respect to $z_n=0$. Namely, we define $u_1(z)=2\,U_2^\varepsilon-u_1(z',-z_n)$ for $z\in B^-_{R_1}$. We still denote these functions by $u_1$ and ${A}$ after the extension. Because of the Dirichlet boundary condition, it is easily seen that $u_1$ satisfies
\begin{equation}\label{eq:u1:ext}
 -\dv_z\big(\mathbf{A}(z,D_z u_1)\big)=\;0 \quad  \text{in} \,\, B_{R_1},
\end{equation}
where the nonlinear operator $\mathbf{A}$ is defined as
$$
\mathbf{A}(z,\xi)=  |A^T\xi|^{p-2} AA^T\xi \quad \quad \text{for} \,\, z\in B_{R_1},\,\,\xi\in \mathbb{R}^n.
$$
By \eqref{co-diff}, similar to \cite{DYZ23}*{Lemma 2.3}, there exists a constant $C=C(n,p,c_1,c_2)>0$, such that for any $z\in B_{R_1}$ and $\xi\in\mathbb{R}^n$,
\begin{align}\label{co-diff2}
    |\mathbf{A}(z,\xi)-|\xi|^{p-2}\xi|\leq C\,|z'|\,|\xi|^{p-1}.
\end{align}
Assume that $r\in(0,R_1]$. We let $v_1\in u_1+W_0^{1,p}(B_r)$ be the unique solution to
\begin{equation}\label{eq:v1}
    \left\{
\begin{aligned}
     -\dv_z(|D_z v_1|^{p-2} D_z v_1) =\;&  0 \quad \;\;\text{in} \,\, B_{r},\\
     v_1 =\; &  u_1 \quad \text{on}\,\, \partial B_{r}. \\
\end{aligned}
\right.
\end{equation}
By testing  \eqref{eq:v1} and \eqref{eq:u1:ext} with $v_1-u_1$ and using \eqref{co-diff2}, we have the comparison estimate
\begin{align}\label{comp1}
    \fint_{B_r}|D_z u_1-D_z v_1|^p\leq Cr^{\min\{2,p\}}\fint_{B_{r}}|D_z u_1|^p,
\end{align}
where $C>0$ is a constant depending only on $n$, $p$, $c_1$, and $c_2$.
For detailed proof of \eqref{comp1}, see \cite{DuzMin10}*{Eq. (4.35)} when $p\in(1,2)$ and \cite{DuzMin11}*{Lemma 3.4} when $p\geq 2$.

Applying Lemma \ref{lem:mean1} and the comparison estimate \eqref{comp1}, we have
\begin{lemma}
Suppose that $u_1\in W^{1,p}(B^+_{R_1})$ is a solution to  \eqref{eq:u1}. Then for any $\mu\in(0,1)$ and $r\in (0,R_1]$, we have
\begin{equation}\label{ineq:pr}
\begin{aligned}
&\left(\fint_{B^+_{ \mu r}}|D_{z'} u_1|^{p}+|D_{z_n}u_1-(D_{z_n} u_1)_{B^+_{ \mu r}}|^{p}\right)^{1/p}\\
&\leq C\mu^{\alpha}\left(\fint_{B^+_{ r}}|D_{z'} u_1|^{p}+|D_{z_n}u_1-(D_{z_n} u_1)_{B^+_{ r}}|^{p}\right)^{1/p}+C_\mu r^{\theta_p}\left(\fint_{B^+_{r}}|D_z u_1|^{p}\right)^{{1}/{p}},
\end{aligned}
\end{equation}
where $\theta_p=\min\{1,2/p\}$, $\alpha$ is the same constant as in Lemma \ref{lem:mean1}, $C_\mu$ is a constant depending on $\mu$, $n$, $p$, $c_1$, and $c_2$, and $C$ is a constant depending on $n$, $p$, $c_1$, and $c_2$.
\end{lemma}
\begin{proof}
 By Lemma \ref{lem:mean1}, \eqref{comp1}, and the triangle inequality, we have
\begin{equation}\label{ineq:phi}
\begin{aligned}
        &\left(\fint_{B_{\mu r}}|D_z u_1-(D_z u_1)_{B_{\mu r}}|^{p}\right)^{1/p}
        \\
        &\leq C\left(\fint_{B_{\mu r}}|D_z v_1-(D_z v_1)_{B_{\mu r}}|^{p}\right)^{1/p}+C\left(\fint_{B_{\mu r}}|D_z u_1-D_z v_1|^{p}\right)^{1/p}\\
        &\leq C\mu^{\alpha}\left(\fint_{B_r}|D_z v_1-(D_z v_1)_{B_{r}}|^{p}\right)^{1/p}+ C\mu^{-\frac n{p}}\left(\fint_{B_r}|D_z u_1-D_z v_1|^{p}\right)^{1/p}\\
        &\leq C\mu^{\alpha}\left(\fint_{B_r}|D_z u_1-(D_z u_1)_{B_{r}}|^{p}\right)^{1/p}+ C\mu^{-\frac n{p}}\left(\fint_{B_r}|D_z u_1-D_z v_1|^{p}\right)^{1/p}\\
        &\leq C\mu^{\alpha}\left(\fint_{B_r}|D_z u_1-(D_z u_1)_{B_{r}}|^{p}\right)^{1/p}+C_\mu r^{\theta_p} \left(\fint_{B_{r}}|D_z u_1|^{p}\right)^{1/p}.
\end{aligned}
\end{equation}
Since $u_1$ is even in $z_n$, \eqref{ineq:phi} directly implies \eqref{ineq:pr}. The proof is completed.
\end{proof}
We now define
\begin{align}\label{def:psi}
\psi(\hat{x},r)=\left(\fint_{\Omega_{r}(\hat{x})}|D_{y'} u|^{p}+|D_{y_n}u-(D_{y_n} u)_{\Omega_r(\hat{x})}|^{p}\right)^{1/p}.
\end{align}
Following a similar argument as in the proof of \cite{DYZ23}*{Lemma 2.5}, we have
\begin{lemma}\label{lem:mean2}
Suppose that $u$ is a solution to \eqref{equinfty} and $\hat{x}\in (\Gamma_+\cup\Gamma_-)\cap \{x\in\mathbb{R}^n: |x'|\leq 3/4\}$. Then there exist constants  $C>0$ depending only on $n$, $p$, $c_1$, and $c_2$, and $C_\mu>0$ depending on $n$, $p$, $c_1$, $c_2$, and $\mu$, such that for any $\mu\in(0,1/4)$ and $r\in(0,\,c_4\underline{\delta}(\hat{x})]$,  it holds that
\begin{equation}\label{ineq:mean2}
\begin{aligned}
       &\psi(\hat{x},\mu r)\leq C\mu^{\alpha}\psi(\hat{x}, r)+C_\mu r^{\theta_p}\left(\fint_{\Omega_{r}(\hat{x})}|D u|^{p}\right)^{1/p},
\end{aligned}
\end{equation}
where $\theta_p=\min\{1,2/p\}$, $\alpha\in(0,1)$ is the same constant as in Lemma \ref{lem:mean1}, $c_4=c_4(n,c_1,c_2)\in(0,\min\{c_1/8,1/8\})$ is the same constant as in \eqref{def:c_4} and $\psi$ is defined in \eqref{def:psi}.
\end{lemma}
By iteration, Lemma \ref{lem:mean2} also implies
\begin{corollary}
    \label{lem:iter1}
Let $u$, $\hat{x}$, $c_4$, $\alpha$, and $\theta_p$ be as in Lemma \ref{lem:mean2} and ${\alpha_1}\in (0,\min\{\alpha, \theta_p\})$. Then there exists a constant  $C>0$ depending only on $n$, $p$, $c_1$, $c_2$, and $\alpha_1$, such that for any  $0<\rho\leq r\leq c_4\underline{\delta}(\hat{x})$, it holds that
\begin{equation*}
\begin{aligned}
    \psi(\hat{x},\rho)\leq  C\left(\frac{\rho}{r}\right)^{\alpha_1} \psi(\hat{x},r)+C\,\rho^{\alpha_1}\|Du\|_{L^\infty(\Omega_{r}(\hat{x}))}.
\end{aligned}
\end{equation*}
\end{corollary}
\begin{proof}
We choose $\mu=\mu(n,p,c_1,c_2,\alpha_1)\in (0,1/6)$ sufficiently small such that $C\mu^{\alpha-{\alpha_1}}<1$, where $C$ is the same constant in \eqref{ineq:mean2}.
Then Lemma \ref{lem:mean2} implies that for any $r\in(0,\,c_4\underline{\delta}(\hat{x})]$, we have
\begin{equation}
                        \label{eq7.15}
         \psi(\hat{x},\mu r)\leq\mu^{\alpha_1}\psi(\hat{x},r)+C_{\alpha_1}r^{\theta_p}\|D u\|_{L^\infty(\Omega_{r}(\hat{x}))},
\end{equation}
where $C_{\alpha_1}>0$ is a constant depending only on $n$, $p$, $c_1$, $c_2$, and $\alpha_1$.
By iteration, from \eqref{eq7.15} we get
\begin{equation}
\begin{aligned}\label{iter:int1}
    &\psi(\hat{x},\mu^j r)\leq \mu^{{\alpha_1} j}\psi(\hat{x},r)+C_{\alpha_1}\sum_{i=1}^j \mu^{{\alpha_1} (i-1)}(\mu^{j-i}r)^{\theta_p}\|Du\|_{L^\infty(\Omega_{r}(\hat{x}))}\\
    &= \mu^{{\alpha_1} j}\psi(\hat{x},r)+C_{\alpha_1}\sum_{i=1}^j \mu^{{\alpha_1} (j-1)}\mu^{(j-i)(\theta_p-\alpha_1)}r^{\theta_p}\|Du\|_{L^\infty(\Omega_{r}(\hat{x}))}\\
    &\leq \mu^{{\alpha_1} j}\psi(\hat{x},r)+C_{\alpha_1}(\mu^j r)^{\alpha_1}\|Du\|_{L^\infty(\Omega_{r}(\hat{x}))}.
\end{aligned}
\end{equation}
Here in the last inequality we used the facts that $\alpha_1<\theta_p$ and $r\in(0,1)$.

Now for any $0<\rho\leq r\leq c_4\underline{\delta}(\hat{x})$, let $j$ be the integer such that
$\mu^{j+1}<{\rho}/{r}\leq \mu^j$. Then by \eqref{iter:int1} with $\mu^{-j}\rho$ in place of $r$, we get
\begin{align*}
      &\psi(\hat{x},\rho)\leq \mu^{{\alpha_1} j}\psi(\hat{x},\mu^{-j}\rho)+C_{ \alpha_1}\,\rho^{\alpha_1}\|Du\|_{L^\infty(\Omega_{\mu^{-j}\rho}(\hat{x}))}\\
    &\leq  C_{\alpha_1}\left(\frac{\rho}{r}\right)^{\alpha_1} \psi(\hat{x},r)+C_{\alpha_1}\,\rho^{\alpha_1}\|Du\|_{L^\infty(\Omega_{r}(\hat{x}))},
\end{align*}
where $C_{\alpha_1}>0$ is a constant depending only on $n$, $p$, $c_1$, $c_2$, and $\alpha_1$. The proof is completed.
\end{proof}

\subsection{Mean oscillation estimates for large \texorpdfstring{$r$}{r}}\label{sec1.4}
Finally, we consider the case when $B_{r}(\Bar{x})$ could potentially intersects with both $\Gamma_+$ and $\Gamma_-$. In this case, we fix $\Bar{x}\in\Omega_{1/2}$ and assume  $\frac{c_4}{216} \underline{\delta}(\Bar{x})\leq r\leq c_5\underline{\delta}(\Bar{x})^{\frac{1}{2}}$, where $\underline{\delta}(\Bar{x})$ is defined in \eqref{def:delta}, $c_4$ is the same constant as in \eqref{def:c_4}, and $c_5$ is a constant which will be determined later.
 We define the map $\mathcal{Z}=\Tilde{\Lambda}_{\Bar{x}}(x)$ by
\begin{equation}\label{Lambda_T}
\left\{
\begin{aligned}
\mathcal{Z}' &= x',\\
\mathcal{Z}_n &= (h_1(\Bar{x}')-h_2(\Bar{x}')+\varepsilon)\Big(\frac{x_n-h_2(x')+{\varepsilon}/{2}}{h_1(x')-h_2(x')+\varepsilon}-\frac{1}{2}\Big).
\end{aligned}
\right.
\end{equation}
Thus $\Tilde{\Lambda}_{\Bar{x}}$ is invertible in ${\Omega}_1$,
$$
{Q}_1:=\Tilde{\Lambda}_{\Bar{x}}({\Omega}_1)=\big\{(\mathcal{Z}',\mathcal{Z}_n)\in \mathbb{R}^n: \,|\mathcal{Z}'|<1, \,|\mathcal{Z}_n|<\frac{1}{2}(h_1(\Bar{x}')-h_2(\Bar{x}')+\varepsilon)\big\},
$$
and
\begin{align*}
     \Tilde{\Gamma}_{\pm}
     :=\Tilde{\Lambda}_{\Bar{x}}(\Gamma_{\pm})
     =\big\{(\mathcal{Z}',\mathcal{Z}_n)\in \mathbb{R}^n: \,|\mathcal{Z}'|<1, \,\mathcal{Z}_n=\pm \frac{1}{2}(h_1(\Bar{x}')-h_2(\Bar{x}')+\varepsilon)\big\}.
\end{align*}
Then $u_2(\mathcal{Z}):=u(\Tilde{\Lambda}_{\Bar{x}}^{-1}(\mathcal{Z}))$ satisfies the following equation with constant Dirichlet boundary conditions
\begin{equation*}
    \left\{
\begin{aligned}
     -\dv_\mathcal{Z}\left(|B^T D_\mathcal{Z} u_2|^{p-2}(\text{det}( B))^{-1}BB^T  D_\mathcal{Z} u_2\right)=\;0 \quad  \text{in} \,\, {Q}_{1}&,\\
  u_2=\;U_1^{\varepsilon}\quad \text{on}\,\, \Tilde{\Gamma}_{+},\quad
   u_2=\;U_2^{\varepsilon}\quad \text{on}\,\, \Tilde{\Gamma}_{-}&,
\end{aligned}
\right.
\end{equation*}
where we denote
$$
B:=B_{\Bar{x}}:=B(\mathcal{Z}):=(b_{ij}(\mathcal{Z})):=D{\Tilde{\Lambda}_{\Bar{x}}}(\Tilde{\Lambda}_{\Bar{x}}^{-1}(\mathcal{Z})).
$$
For $\mathcal{Z}\in Q_{1}$, let $x=\Tilde{\Lambda}_{\Bar{x}}^{-1}(\mathcal{Z})$.
 Then
$$b_{ii}(\mathcal{Z})=1 \quad \text{for}  \;i\in\{1,2,\ldots,n-1\},$$
$$ b_{ij}(\mathcal{Z})=0 \quad \text{for} \; i\in\{1,2,\ldots,n-1\},\;j\in\{1,2,\ldots,n\},\; i\neq j,
$$
\begin{align*}
    b_{nj}(\mathcal{Z})&=\frac{h_1(\Bar{x}')-h_2(\Bar{x}')+\varepsilon}
    {(h_1(x')-h_2(x')+\varepsilon)^2}\\
    &\quad \cdot \big[D_{x_j}h_2(x')\big(x_n-h_1(x')-\frac{\varepsilon}{2}\big)-D_{x_j} h_1(x')\big(x_n-h_2(x')+\frac{\varepsilon}{2}\big)\big]
\end{align*}
for $j\in\{1,2,\ldots,n-1\}$, and
$$
b_{nn}(\mathcal{Z})=\frac{h_1(\Bar{x}')-h_2(\Bar{x}')+\varepsilon}{h_1(x')-h_2(x')+\varepsilon}.
$$
Therefore,
\begin{align*}
    \text{det}(B(\mathcal{Z}))=b_{nn}(\mathcal{Z})=\frac{h_1(\Bar{x}')-h_2(\Bar{x}')+\varepsilon}{h_1(\mathcal{Z}')-h_2(\mathcal{Z}')+\varepsilon}
\end{align*}
is a function independent of $\mathcal{Z}_n$.
Assume
\begin{equation}
                        \label{eq11.05}
\frac{c_4}{216} \underline{\delta}(\Bar{x})\leq r\leq \frac{1}{8}\underline{\delta}(\Bar{x})^{\frac{1}{2}}
\end{equation}
and let $\Bar{\mathcal{Z}}= \Tilde{\Lambda}_{\Bar{x}}(\Bar{x})$.
When $\mathcal{Z}\in\mathbb{R}^n$ satisfies $|\mathcal{Z}'-\Bar{\mathcal{Z}}'|\leq r$, by the triangle inequality and \eqref{eq11.05}, we always have
\begin{align}\label{Z_tan}
    |Z'|\leq |\Bar{\mathcal{Z}}'|+r \leq |\Bar{x}'|+r<1.
\end{align}
Then for any $\mathcal{Z}\in Q_{1}$ with $|\mathcal{Z}'-\Bar{\mathcal{Z}}'|\leq r$ and $x=\Tilde{\Lambda}_{\Bar{x}}^{-1}(\mathcal{Z})$, by the triangle inequality and \eqref{eq11.05}, we have
$$
|x'|\leq r+|\Bar{x}'|\leq \big(1+\sqrt{216/c_4}\big)r^{\frac{1}{2}} \quad \text{and} \quad |x'|^2\geq \frac{1}{2}|\Bar{x}'|^2-r^2\geq \frac{1}{4} (|\Bar{x}'|^2-\varepsilon).
$$
Thus, using \eqref{fg_0}, \eqref{fg_1}, and \eqref{def:c_2}, we infer that for $j=1,2,\ldots,n-1$ and some constants $C>0$ depending only on $n$, $p$, $c_1$, and $c_2$,
\begin{equation*}
    \begin{aligned}
    &|b_{nj}(\mathcal{Z})|\leq 2c_2\frac{|x'|(h_1(\Bar{x}')-h_2(\Bar{x}')+\varepsilon)}{h_1(x')-h_2(x')+\varepsilon}\leq 2c_2 \frac{|x'|(2c_2|\Bar{x}'|^2 +\varepsilon)}{c_1|x'|^2+\varepsilon}
    \\&\leq C|x'|\leq Cr^{\frac{1}{2}}
    \leq \frac{Cr}{\underline{\delta}(\Bar{x})^\frac{1}{2}},
\end{aligned}
\end{equation*}
\begin{equation}\label{B-I-comp}
    \begin{aligned}
    &|b_{nn}(\mathcal{Z})-1|=\Big|\frac{\int_0^1 \frac{d}{dt}\big(h_1(tx'+(1-t)\Bar{x}')-h_2(tx'+(1-t)\Bar{x}')\big)dt}{h_1(x')-h_2(x')+\varepsilon}\Big|\\
    &\leq 2c_2\frac{(|x'|+|\Bar{x}'|)|x'-\Bar{x}'|}{c_1|x'|^2+\varepsilon}\leq \frac{Cr}{\underline{\delta}(\Bar{x})^\frac{1}{2}},
\end{aligned}
\end{equation}
and similarly,
\begin{equation}\label{det-1}
\begin{aligned}
 &\big|\big(\text{det}(B(\mathcal{Z}))\big)^{-1}-1\big|=\big|\big(b_{nn}(\mathcal{Z})\big)^{-1}-1\big|\leq 
 \frac{Cr}{\underline{\delta}(\Bar{x})^\frac{1}{2}}.
\end{aligned}
\end{equation}
Therefore, when \eqref{eq11.05} holds and $\mathcal{Z}\in Q_{1}$ with $|\mathcal{Z}'-\Bar{\mathcal{Z}}'|\leq r$, we have for some constant $C=C(n,p,c_1,c_2)>0$,
\begin{equation}\label{co-diff3}
   |B(\mathcal{Z})-I_n|\leq \frac{Cr}{\underline{\delta}(\Bar{x})^\frac{1}{2}}.
\end{equation}
In particular, there exists a constant
\begin{equation}\label{def:c5}
    c_5=c_5(n,p,c_1,c_2)\in(0,1/8),
\end{equation}
such that if $\frac{c_4}{216} \underline{\delta}(\Bar{x})\leq c_5\underline{\delta}(\Bar{x})^{\frac{1}{2}}$
and $\mathcal{Z}\in Q_{1}$ with $|\mathcal{Z}'-\Bar{\mathcal{Z}}'|\leq c_5\underline{\delta}(\Bar{x})^{\frac{1}{2}}$, it also holds that
\begin{align}\label{co-diff4}
    |B(\mathcal{Z})-I_n|\leq 1/2\quad \text{and} \quad \big|(\text{det}(B(\mathcal{Z})))^{-1}-1\big|\leq 1/2.
\end{align}
Next we extend $u_2$ and ${B}$ to the whole cylinder $\mathcal{C}_{1}:=\{(\mathcal{Z}',\mathcal{Z}_n)\in \mathbb{R}^n: \,|\mathcal{Z}'|<1\}$. We denote $H:=h_1(\Bar{x}')-h_2(\Bar{x}')+\varepsilon$.
We take the even extension of $b_{nn}$, and $b_{ij}, i,j = 1,2,\ldots, n-1,$ with respect to $\mathcal{Z}_n =H/2$, and take the odd extension of $b_{in}$ and $b_{ni}, i = 1,2,\ldots, n-1,$ with respect to $\mathcal{Z}_n =H/2$. Then we take the periodic extension of $B$ in  the $\mathcal{Z}_n$ axis, so that the period is equal to $2H$.
Then we inductively reflect $u_2$ with respect to $\pm(k+\frac{1}{2})H$ for $k\in \mathbb{N}$. Namely, for $\mathcal{Z}\in\mathcal{C}_1$ and $k\in\mathbb{Z}$, we define
\begin{equation*}
   u_2(\mathcal{Z}):=\left\{
   \begin{aligned}
       & 2kU_1^{\varepsilon}-2kU_2^{\varepsilon}+u_2(\mathcal{Z}_n-2kH), && \mbox{if }|\mathcal{Z}_n-2kH|\leq  \frac{H}{2},\\
       & (2k+2)U_1^{\varepsilon}-2kU_2^{\varepsilon}-u_2((2k+1)H-\mathcal{Z}_n), && \mbox{if } |\mathcal{Z}_n-(2k+1)H|\leq  \frac{H}{2}.
    \end{aligned}\right.
\end{equation*} Then because of the Dirichlet boundary conditions, it is easily seen that $u_2$ satisfies
\begin{equation}\label{eq:u3}
 -\dv_\mathcal{Z}\big(\mathbf{B}(\mathcal{Z},D_\mathcal{Z} u_2)\big)=\;0 \quad  \text{in} \;\mathcal{C}_{1},
\end{equation}
where the nonlinear operator $\mathbf{B}$ is defined as
$$
\mathbf{B}(\mathcal{Z},\xi)= (\text{det}(B(\mathcal{Z})))^{-1}| B^T \xi|^{p-2} BB^T\xi  \quad \quad \text{for} \,\, \mathcal{Z}\in\mathcal{C}_{1},\,\xi\in \mathbb{R}^n,
$$
and
$$(\text{det}(B(\mathcal{Z})))^{-1}=\big(b_{nn}(\mathcal{Z})\big)^{-1}=\frac{h_1(\mathcal{Z}')-h_2(\mathcal{Z}')+\varepsilon}{h_1(\Bar{x}')-h_2(\Bar{x}')+\varepsilon}.
$$
Similar to \eqref{co-diff2}, using \eqref{Z_tan}, \eqref{det-1}, \eqref{co-diff3}, \eqref{co-diff4}, and \eqref{Z_tan}, we obtain that  for any $r\in \big[\frac{c_4}{216} \underline{\delta}(\Bar{x}),c_5\underline{\delta}(\Bar{x})^{\frac{1}{2}}\big]$, $\mathcal{Z}\in B_{r}(\Bar{\mathcal{Z}})$, and $\xi\in\mathbb{R}^n$,
\begin{align}\label{co-diff5}
    |\mathbf{B}(\mathcal{Z},\xi)-|\xi|^{p-2}\xi|\leq \frac{Cr}{\underline{\delta}(\Bar{x})^\frac{1}{2}}|\xi|^{p-1},
\end{align}
where $C>0$ is a constant depending only on $n$, $p$, $c_1$, and $c_2$.
Now we let $v_2\in u_2+W_0^{1,p}(B_r(\Bar{\mathcal{Z}}))$ be the unique solution to

\begin{equation*}
    \left\{
\begin{aligned}
     -\dv_\mathcal{Z}\big(|D_\mathcal{Z} v_2|^{p-2} D_\mathcal{Z} v_2\big) =\;&  0 \quad \;\;\text{in} \,\, B_{r}(\Bar{\mathcal{Z}}),  \\
     v_2 =\; &  u_2 \quad \text{on}\,\, \partial B_{r}(\Bar{\mathcal{Z}}).
\end{aligned}
\right.
\end{equation*}
Using \eqref{co-diff5}, similar to \eqref{comp1}, we have the following comparison estimate
\begin{align}\label{comp2}
    \fint_{B_r(\Bar{\mathcal{Z}})}|D_\mathcal{Z} u_2-D_\mathcal{Z} v_2|^p\leq C\Big(\frac{r}{\underline{\delta}(\Bar{x})^\frac{1}{2}}\Big)^{\min\{2,p\}}\fint_{B_{r}(\Bar{\mathcal{Z}})}|D_\mathcal{Z} u_2|^p,
\end{align}
where $C>0$ is a constant depending only on $n$, $p$, $c_1$, and $c_2$.

For $\Bar{x}\in \Omega_{1/2}$ and $r\in (0, \frac{1}{8}\underline{\delta}(\Bar{x})^{\frac{1}{2}})$, we define
\begin{align}\label{def:phi2}
\Tilde{\phi}(\Bar{x},r)=\left(\fint_{B_{r}(\Bar{\mathcal{Z}})}|D_{\mathcal{Z}} u_2-(D_{\mathcal{Z}} u_2)_{B_r(\Bar{\mathcal{Z}})}|^{p}\right)^{1/p}.
\end{align}
Then following the same proofs as those of Lemma \ref{lem:mean2} and Corollary \ref{lem:iter1} with \eqref{comp2} in place of \eqref{comp1}, we have
\begin{lemma}\label{lem:mean3}
Suppose that $\Bar{x}\in\Omega_{1/2}$ and  $u_2$ is a solution to \eqref{eq:u3}. Then there exist constants $C>0$ depending only on $n$, $p$, $c_1$, and $c_2$, and $C_\mu>0$ depending on $n$, $p$, $c_1$, $c_2$, and $\mu$, such that for any $\mu\in(0,1)$ and $r\in[\frac{c_4}{216}\underline{\delta}(\Bar{x}),\,c_5\underline{\delta}(\Bar{x})^{\frac{1}{2}}]$, it holds that
\begin{equation*}
\begin{aligned}
\Tilde{\phi}(\Bar{x},\mu r)\leq C\mu^\alpha \Tilde{\phi}(\Bar{x},r)+C_\mu\Big(\frac{ r}{\underline{\delta}(\Bar{x})^\frac{1}{2}}\Big)^{\theta_p}\left(\fint_{B_{r}(\Bar{\mathcal{Z}})}|D_\mathcal{Z} u_2|^{p}\right)^{1/p},
\end{aligned}
\end{equation*}
where $\theta_p=\min\{1,2/p\}$, $\alpha$ is the same constant as in Lemma \ref{lem:mean1}, $c_4$, $c_5$ are the same constants as in \eqref{def:c_4} and \eqref{def:c5}, and $\Tilde{\phi}$ is defined in \eqref{def:phi2}.
Moreover, for any $\alpha_1\in (0,\min\{\alpha, \theta_p\})$, there exists a constant  $C>0$ depending only on $n$, $p$, $c_1$, $c_2$, and $\alpha_1$, such that for any  $\frac{c_4}{216}\underline{\delta}(\Bar{x})\leq\rho\leq r\leq c_5\underline{\delta}(\Bar{x})^{\frac{1}{2}}$, it holds that
\begin{equation}\label{iter:int3}
\begin{aligned}
    \Tilde{\phi}(\Bar{x},\rho)&\leq  C\left(\frac{\rho}{r}\right)^{\alpha_1} \Tilde{\phi}(\Bar{x},r)+C\,\Big(\frac{\rho}{\underline{\delta}(\Bar{x})^\frac{1}{2}}\Big)^{\alpha_1}\|D_\mathcal{Z} u_2\|_{L^\infty({B_{r}(\Bar{\mathcal{Z}})})}.
\end{aligned}
\end{equation}
\end{lemma}
\subsection{Mean oscillation decay estimates}
Now we deduce mean oscillation decay estimates by connecting the three different cases of radii $r$, when $\underline{\delta}(\Bar{x})$ is sufficiently small.
We let
\begin{equation}\label{def:alpha1}
    \alpha_1=\frac{1}{2}\min\{\alpha, \frac{2}{p}\}
\end{equation}
and assume
\begin{equation}\label{small-assump}
    \underline{\delta}(\Bar{x})\leq \min\big\{(\frac{c_5}{10c_4})^2,\frac{1}{4+4c_2}\big\},
\end{equation}
where $\alpha$, $c_4$, and $c_5$ are the same constants as in Lemmas \ref{lem:mean1}, \ref{lem:mean2}, and \ref{lem:mean3}, respectively.

Let $\Bar{x}\in \Omega_{1/2}$.
By \eqref{fg_0}, \eqref{def:c_2} and \eqref{small-assump}, we have
\begin{equation}\label{dist-bound}
    \dist(\Bar{x},\Gamma_+\cup \Gamma_-)\leq \frac{1}{2} (h_1(\Bar{x}')-h_2(\Bar{x}')+\varepsilon)\leq c_2|\Bar{x}'|^2+\varepsilon\leq 1/4.
\end{equation}and thus
\begin{equation}\label{dist-int}
    \dist(\Bar{x},\Gamma_+\cup \Gamma_-)=\dist(\Bar{x},\partial \Omega_1).
\end{equation}
\begin{lemma}\label{lem:iter}
Let $\Bar{x}\in\Omega_{1/2}$,  $u$ be a solution to \eqref{equinfty}, and $r=\frac{1}{8}\underline{\delta}(\Bar{x})^{1/2}$.
 Let $\phi$ and $\Tilde{\phi}$ be defined as in \eqref{def:phi} and \eqref{def:phi2} and let $\alpha_1$ and $c_4$ be the same constants as in \eqref{def:alpha1} and \eqref{def:c_4}. Assume that \eqref{small-assump} holds. Then there exists a constant $C>0$  depending only on  $n$, $p$, $c_1$, and $c_2$, such that the following holds:
\begin{itemize}
    \item[(i)] For any $\rho\in(0,\frac{c_4}{6}\underline{\delta}(\Bar{x})]$,
    \begin{equation}\label{ineq:global1}
    \phi(\Bar{x},\rho)\leq  C\left(\frac{\rho}{r}\right)^{\alpha_1}\|D u\|_{L^{\infty}(\Omega_{\Bar{x}, r})}.
\end{equation}
   \item[(ii)] For any  $\rho\in[\frac{c_4}{216}\underline{\delta}(\Bar{x}),r]$,
   \begin{equation}\label{ineq:global2}
    \Tilde{\phi}(\Bar{x},\rho)\leq  C\left(\frac{\rho}{r}\right)^{\alpha_1}\|D u\|_{L^{\infty}(\Omega_{\Bar{x}, r})}.
\end{equation}
\end{itemize}
\end{lemma}
\begin{proof}
First, we prove assertion (ii). Note that $B_r(\Bar{\mathcal{Z}})\cap Q_1\subset \Tilde{\Lambda}_{\Bar{x}}(\Omega_{\Bar{x},r})$. By the definition of the extended solution $u_2$, \eqref{ineq:global2} clearly holds for $\rho\in [c_5\underline{\delta}(\Bar{x})^{1/2},r]$, where $c_5\in(0,1/8)$ is the same constant as in \eqref{def:c5}. On the other hand, \eqref{iter:int3} directly implies
\eqref{ineq:global2} for $\rho\in [\frac{c_4}{216}\underline{\delta}(\Bar{x}), c_5\underline{\delta}(\Bar{x})^{1/2})$.

Next, we give the proof of assertion (i). We consider the following three cases:
\begin{equation*}
    \begin{aligned}
        &\dist(\Bar{x},\Gamma_+\cup \Gamma_-)\leq \rho\leq \frac{c_4}{6}\underline{\delta}(\Bar{x}),\\
        &\rho<\dist(\Bar{x},\Gamma_+\cup \Gamma_-)\leq \frac{c_4}{6}\underline{\delta}(\Bar{x}),\\
        &\rho\leq \frac{c_4}{6}\underline{\delta}(\Bar{x})<\dist(\Bar{x},\Gamma_+\cup \Gamma_-).
    \end{aligned}
\end{equation*}

{\em Case 1: $\dist(\Bar{x},\Gamma_+\cup \Gamma_-)\leq \rho\leq \frac{c_4}{6}\underline{\delta}(\Bar{x})$.}
Since $\Bar{x}\in\Omega_{1/2}$, by \eqref{dist-bound} and the triangle inequality, we can choose $\hat{x}\in \Gamma_+\cup\Gamma_-$ with $|\hat{x}|\leq 3/4$, such that $\dist(\Bar{x},\Gamma_+\cup\Gamma_-)=|\hat{x}-\Bar{x}|$, and thus $\Omega_{\rho}(\Bar{x})\subset \Omega_{2\rho}(\hat{x})\subset \Omega_{3\rho}(\Bar{x})$.

Since $|\hat{x}-\Bar{x}|\leq \rho\leq \frac{c_4}{6}\underline{\delta}(\Bar{x})$, by the triangle inequality, we have
\begin{align*}
    |\hat{x}'|^2\leq  2|\Bar{x}'|^2+2\big(\frac{c_4}{6}\big)^2\underline{\delta}(\Bar{x})^2, \quad |\Bar{x}'|^2\leq  2|\hat{x}'|^2+2\big(\frac{c_4}{6}\big)^2\underline{\delta}(\Bar{x})^2,
\end{align*}
which also implies
\begin{align}\label{deltahat}
   \frac{1}{3}\underline{\delta}(\Bar{x})\leq \underline{\delta}(\hat{x}) \leq 3\underline{\delta}(\Bar{x})
\end{align}
 since $c_4\in (0,1)$.
 By \eqref{deltahat} and the fact that $|\hat{x}-\Bar{x}|\leq\rho\leq \frac{c_4}{6}\underline{\delta}(\Bar{x})$, we also have
\begin{align}\label{cond:mean21}
    2|\hat{x}-\Bar{x}|\leq 2\rho\leq c_4\underline{\delta}(\hat{x}).
\end{align}
Let $R_1=c_4\underline{\delta}(\hat{x})$. Then $\Omega_{R_1}(\hat{x})\subset \Omega_{\frac{3}{2}R_1}(\Bar{x})$.
Thus we can apply Corollary \ref{lem:iter1} at
 $\hat{x}\in(\Gamma_+\cup\Gamma_-)\cap\{x\in\mathbb{R}^n: |x'|\leq 3/4\}$ and use \eqref{volume1} to obtain
 \begin{equation}\label{case1-1}
 \begin{aligned}
      &\phi(\Bar{x},\rho)\leq C\psi(\hat{x},2\rho)\leq  C\left(\frac{\rho}{R_1}\right)^{\alpha_1} \psi(\hat{x},R_1)+C\rho^{\alpha_1}\|Du\|_{L^\infty(\Omega_{R_1}(\hat{x}))}\\
      &\leq C\left(\frac{\rho}{R_1}\right)^{\alpha_1} \psi(\hat{x},R_1)+C\rho^{\alpha_1}\|Du\|_{L^\infty(\Omega_{\frac{3}{2}R_1}(\Bar{x}))}.
 \end{aligned}
 \end{equation}
By using \eqref{rotation} and the change of variables $x\rightarrow y$, we have
\begin{equation}\label{case1-2}
\begin{aligned}
    &\psi(\hat{x}, R_1)\leq \left(\fint_{\Omega_{R_1}(\hat{x})}|D_{x'} u|^{p}+|D_{x_n}u-(D_{x_n} u)_{\Omega_{R_1}(\hat{x})}|^{p}\right)^{1/p}
    \\&\quad+C|\hat{x}'|\,\|Du\|_{L^\infty(\Omega_{R_1}(\hat{x}))}
    \\&\leq \left(\fint_{\Omega_{R_1}(\hat{x})}|D_{x'} u|^{p}+|D_{x_n}u-(D_{x_n} u)_{\Omega_{R_1}(\hat{x})}|^{p}\right)^{1/p}
    \\&\quad+CR_1^{1/2}\,\|Du\|_{L^\infty(\Omega_{\frac{3}{2}R_1}(\Bar{x}))}.
\end{aligned}
\end{equation}
By \eqref{deltahat}, \eqref{small-assump}, and the fact that  $c_5\in(0,1/8)$, it holds that
\begin{equation}\label{R1bound}
   \frac{c_4}{2}  \underline{\delta}(\Bar{x})\leq \frac{3}{2}R_1\leq  \frac{9c_4}{2}\underline{\delta}(\Bar{x})\leq \frac{c_5}{2} \underline{\delta}(\Bar{x})^{1/2}\leq \frac{1}{2}r.
\end{equation}
Since $\Omega_{R_1}(\hat{x})\subset \Omega_{\frac{3}{2}R_1}(\Bar{x})$, by using \eqref{volume1}, \eqref{det-1}--\eqref{co-diff4}, \eqref{R1bound}, the change of variables $x\rightarrow \mathcal{Z}$, and the triangle inequality, we also have
\begin{equation}
    \begin{aligned}\label{case1-3}
        &\left(\fint_{\Omega_{R_1}(\hat{x})}|D_{x'} u|^{p}+|D_{x_n}u-(D_{x_n} u)_{\Omega_{R_1}(\hat{x})}|^{p}\right)^{1/p}\\
        &\leq C\frac{|\Tilde{\Lambda}_{\Bar{x}}(\Omega_{R_1}(\hat{x}))|^{1/p}}{|\Omega_{R_1}(\hat{x})|^{1/p}}\left(\fint_{\Tilde{\Lambda}_{\Bar{x}}(\Omega_{R_1}(\hat{x}))}|D_{\mathcal{Z}'} u_2|^{p}+|D_{\mathcal{Z}_n}u_2-(D_{\mathcal{Z}_n} u_2)_{\Tilde{\Lambda}_{\Bar{x}}(\Omega_{R_1}(\hat{x}))}|^{p}\right)^{1/p}
        \\&\quad+ C\, \frac{R_1}{\underline{\delta}(\Bar{x})^{\frac{1}{2}}} \|Du\|_{L^\infty(\Omega_{\frac{3}{2}R_1}(\Bar{x}))}\\
        & \leq C\left(\fint_{\Tilde{\Lambda}_{\Bar{x}}(\Omega_{R_1}(\hat{x}))}|D_{\mathcal{Z}'} u_2|^{p}+|D_{\mathcal{Z}_n}u_2-(D_{\mathcal{Z}_n} u_2)_{\Tilde{\Lambda}_{\Bar{x}}(\Omega_{R_1}(\hat{x}))}|^{p}\right)^{1/p}
        \\&\quad+ CR_1^{1/2}\|Du\|_{L^\infty(\Omega_{\frac{3}{2}R_1}(\Bar{x}))}.
    \end{aligned}
\end{equation}
Without loss of generality, we assume $\hat{x}\in \Gamma_-$ and thus $\hat{\mathcal{Z}}:=\Tilde{\Lambda}_{\Bar{x}}(\hat{x})\in \Tilde{\Gamma}_-$.
We denote $B_{R}^+(\hat{\mathcal{Z}}):=B_{R}(\hat{\mathcal{Z}})\cap\{\mathcal{Z}\in\mathbb{R}^n: \mathcal{Z}_n>\hat{\mathcal{Z}}_n\}$ for any $R>0$.
  By \eqref{co-diff4} and \eqref{cond:mean21}, we have $B_{2R_1}(\hat{\mathcal{Z}})\subset B_{3R_1}(\Bar{\mathcal{Z}})\subset B_{r}(\Bar{\mathcal{Z}})\subset \mathcal{C}_{1}=\{(\mathcal{Z}',\mathcal{Z}_n)\in \mathbb{R}^n: \,|\mathcal{Z}'|<1\}$.
Since $c_4\leq \min\{c_1/8,1/8\}$, by \eqref{fg_1} and \eqref{deltahat}, we know that
$B_{2R_1}(\hat{\mathcal{Z}})\cap Q_{1}\equiv B_{2R_1}^+(\hat{\mathcal{Z}})$.
Again by \eqref{co-diff4}, we also have $ B^+_{R_1/2}(\hat{\mathcal{Z}})\subset\Tilde{\Lambda}_{\Bar{x}}(\Omega_{R_1}(\hat{x}))\subset B^+_{2R_1}(\hat{\mathcal{Z}})$. Therefore, by the triangle inequality and the definition of $u_2$,
\begin{equation}\label{case1-4}
\begin{aligned}
   &\left(\fint_{\Tilde{\Lambda}_{\Bar{x}}(\Omega_{R_1}(\hat{x}))}|D_{\mathcal{Z}'} u_2|^{p}+|D_{\mathcal{Z}_n}u_2-(D_{\mathcal{Z}_n} u_2)_{\Tilde{\Lambda}_{\Bar{x}}(\Omega_{R_1}(\hat{x}))}|^{p}\right)^{1/p}\\
   &\leq C\left(\fint_{B^+_{2R_1}(\hat{\mathcal{Z}})}|D_{\mathcal{Z}'} u_2|^{p}+|D_{\mathcal{Z}_n}u_2-(D_{\mathcal{Z}_n} u_2)_{B^+_{2R_1}(\hat{\mathcal{Z}})}|^{p}\right)^{1/p}\\
   &= C \left(\fint_{B_{2R_1}(\hat{\mathcal{Z}})}|D_{\mathcal{Z}}u_2-(D_{\mathcal{Z}} u_2)_{B_{2R_1}(\hat{\mathcal{Z}}_0)}|^{p}\right)^{1/p}
   \leq C \Tilde{\phi}(\Bar{x},3R_1).
\end{aligned}
\end{equation}
Combining \eqref{case1-1}, \eqref{case1-2}, \eqref{case1-3}, and \eqref{case1-4}, we have
\begin{equation}\label{case1-5}
\begin{aligned}
    &\phi(\Bar{x},\rho)
    \\&\leq C\left(\frac{\rho}{R_1}\right)^{\alpha_1} \Big(\Tilde{\phi}(\Bar{x},3R_1)+R_1^{1/2}\|Du\|_{L^\infty(\Omega_{\frac{3}{2}R_1}(\Bar{x}))}\Big)+C\,\rho^{\alpha_1}\|Du\|_{L^\infty(\Omega_{\frac{3}{2}R_1}(\Bar{x}))}.
\end{aligned}
\end{equation}
Note that by \eqref{deltahat}, $R_1^{-\alpha_1}R_1^{1/2}\leq R_1^{-\alpha_1/2}\leq Cr^{-\alpha_1}$. Thus \eqref{ineq:global2} with $3R_1$ in place of $\rho$ and \eqref{case1-5} directly imply \eqref{ineq:global1}.

{\em Case 2: $\rho<\dist(\Bar{x},\Gamma_+\cup \Gamma_-)\leq \frac{c_4}{6}\underline{\delta}(\Bar{x})$.} Let $R_2=\dist(\Bar{x},\Gamma_+\cup \Gamma_-)$. Then we can apply the estimate \eqref{ineq:global1} in Case 1 with $R_2$ in place of $\rho$ to obtain
\begin{align}\label{case2-1}
    \phi(\Bar{x},R_2)\leq  C\left(\frac{R_2}{r}\right)^{\alpha_1}\|D u\|_{L^{\infty}(\Omega_{\Bar{x},r})}.
\end{align}
By \eqref{dist-int}, $B_{R_2}(\Bar{x})\subset \Omega_1$, and we can apply Lemma \ref{lem:mean1} to get
\begin{align}\label{case2-2}
    \phi(\Bar{x},\rho)\leq  C\left(\frac{\rho}{R_2}\right)^{\alpha_1} \phi(\Bar{x},R_2).
\end{align}
Combining \eqref{case2-1} and \eqref{case2-2} yields \eqref{ineq:global1}.

{\em Case 3: $\rho\leq \frac{c_4}{6}\underline{\delta}(\Bar{x})<\dist(\Bar{x},\Gamma_+\cup \Gamma_-)$.} Let $R_3=\frac{c_4}{6}\underline{\delta}(\Bar{x})$. Then by \eqref{dist-int}, $B_{R_3}(\Bar{x})\subset \Omega_1$.
By Lemma \ref{lem:mean1}, we have
\begin{align}\label{case3-1}
    \phi(\Bar{x},\rho)\leq  C\left(\frac{\rho}{R_3}\right)^{\alpha_1} \phi(\Bar{x},R_3).
\end{align}
Similar to Case 1, by using \eqref{det-1}--\eqref{co-diff4}, change of variables, and the triangle inequality, we obtain
\begin{equation}\label{case3-2}
\begin{aligned}
    \phi(\Bar{x},R_3) &\leq C\frac{|\Tilde{\Lambda}_{\Bar{x}}(B_{R_3}(\Bar{x}))|^{1/p}}{|B_{R_3}(\Bar{x})|^{1/p}}\left(\fint_{\Tilde{\Lambda}_{\Bar{x}}(B_{R_3}(\Bar{x}))}|D_{\mathcal{Z}} u_2-(D_{\mathcal{Z}} u_2)_{\Tilde{\Lambda}_{\Bar{x}}(B_{R_3}(\Bar{x}))}|^{p}\right)^{1/p}\\
    &\quad+
    C \,\frac{R_3}{\underline{\delta}(\Bar{x})^{\frac{1}{2}}} \|Du\|_{L^\infty(B_{R_3}(\Bar{x}))}
     \\&  \leq C\Tilde{\phi}(\Bar{x}, 2R_3)+ \frac{CR_3}{r} \, \|Du\|_{L^\infty(B_{R_3}(\Bar{x}))}.
\end{aligned}
\end{equation}
By \eqref{ineq:global2} with $2R_3$ in place of $\rho$, we also have
\begin{align}\label{case3-3}
    \Tilde{\phi}(\Bar{x},2R_3)\leq  C\left(\frac{R_3}{r}\right)^{\alpha_1}\|D u\|_{L^{\infty}(\Omega_{\Bar{x},r})}.
\end{align}
Combining \eqref{case3-1}, \eqref{case3-2}, and \eqref{case3-3} yields \eqref{ineq:global1}.
The proof is completed.
\end{proof}
  It is straightforward to see the following lower bound of $|\Omega_\rho(\Bar{x})|$ from the proof of Lemma \ref{lem:iter} (assertion (i), case 1) and \eqref{volume1}, which would be useful in the proof of Proposition \ref{thm-1/2}.
\begin{lemma}\label{lem:volume}
Let $\Bar{x}\in\Omega_{1/2}$ and $c_4=c_4(n,c_1,c_2)\in(0,1)$ be the same constant as in \eqref{def:c_4}. Assume that \eqref{small-assump} holds. Then there exists a constant $c>0$  depending only on  $n$, such that for any
$\rho\in(0,\frac{c_4}{6}\underline{\delta}(\Bar{x}))$, it holds that
\begin{equation*}
    |\Omega_\rho(\Bar{x})|\ge c\rho^n.
\end{equation*}
\end{lemma}

\subsection{Proof of Proposition \ref{thm-1/2}}
Now we are ready to prove Proposition \ref{thm-1/2}.

\begin{proof}[Proof of Proposition \ref{thm-1/2}]
Let $x_0\in\Omega_{1/4}$ and we prove the proposition around $x=x_0$.
We recall $\underline{\delta}(x_0)=\varepsilon+|x_0'|^2$. Let $x\in \Omega_{x_0,\sqrt{\underline{\delta}(x_0)}/4}$.
By the triangle inequality,
\begin{align*}
    |x'|^2\leq  2|x_0'|^2+\underline{\delta}(x_0)/8, \quad |x_0'|^2\leq 2|x'|^2+\underline{\delta}(x_0)/8.
\end{align*}
Therefore, for any $x\in \Omega_{x_0,\sqrt{\underline{\delta}(x_0)}/4}$, it holds that
\begin{align}\label{delta1}
    \frac{\underline{\delta}(x_0)}{3}\leq \underline{\delta}(x) \leq 3\underline{\delta}(x_0).
\end{align}
We denote
\begin{equation*}
    c_0:=\frac{1}{3}\min\big\{(\frac{c_5}{10c_4})^2,\frac{1}{4+4c_2}\big\},
\end{equation*}
where $c_4$ and $c_5$ are the same constants as in \eqref{def:c_4} and \eqref{def:c5}. Thus $c_0>0$ depends only on $n$, $p$, $c_1$, and $c_2$.
When $\underline{\delta}(x_0)>c_0$, by \eqref{delta1}, we can apply classical results of H\"older regularity of the gradient for the $p$-Laplace equation to get \eqref{gradient-1/2} with $x=x_0$, and some $\beta=\beta(n,p)\in(0,1)$ and $C=C(n,p,c_1,c_2)>0$.
Let $x_1, x_2\in\Omega_{x_0,\sqrt{\underline{\delta}(x_0)}/4}$ and we denote
$$\rho:=|x_1-x_2|.$$
It suffices to show that when $\underline{\delta}(x_0)\leq c_0$,
\begin{align}\label{c1alpha}
    |Du(x_1)-Du(x_2)|\leq C \underline{\delta}(x_0)^{-\alpha_1/2} \rho^{\alpha_1}\|Du\|_{L^\infty(\Omega_{x_0,\sqrt{\underline{\delta}(x_0)}/2})}
\end{align}
holds for some constant $C>0$ depending only on $n$, $p$, $c_1$, and $c_2$, where $\alpha_1\in(0,1)$ is the same constant as in \eqref{def:alpha1}.
From now on, we assume
\begin{equation}\label{deltasmall}
   \underline{\delta}(x_0)\leq c_0=\frac{1}{3}\min\big\{(\frac{c_5}{10c_4})^2,\frac{1}{4+4c_2}\big\}.
\end{equation}
By \eqref{delta1}, we have
\begin{align}\label{delta2}
    \frac{\underline{\delta}(x_0)}{3}\leq \underline{\delta}(x_1) \leq 3\underline{\delta}(x_0)\quad \mbox{and}\quad \frac{\underline{\delta}(x_0)}{3}\leq \underline{\delta}(x_2) \leq 3\underline{\delta}(x_0)
\end{align}
and therefore \eqref{deltasmall} also implies that
\eqref{small-assump} holds for $\Bar{x}=x_1, x_2\in \Omega_{x_0,\sqrt{\underline{\delta}(x_0)}/4}\subset \Omega_{1/2}$.
Thus, we can apply  both Lemmas \ref{lem:iter} and \ref{lem:volume} with $\Bar{x}=x_1$ and with $\Bar{x}=x_2$.

We consider two different cases: $\rho\leq \frac{c_4}{36}\underline{\delta}(x_0)$ and $\rho>\frac{c_4}{36}\underline{\delta}(x_0)$.

{\em Case 1: $\rho\leq \frac{c_4}{36}\underline{\delta}(x_0)$.}
By \eqref{delta2}, we also have
\begin{equation}\label{small-rho}
    \rho\leq \frac{c_4}{12}\underline{\delta}(x_1)\quad \mbox{and} \quad \rho\leq \frac{c_4}{12}\underline{\delta}(x_2).
\end{equation}
For any $x\in \Omega_\rho(x_2)$, by the triangle inequality
\begin{align*}
    | Du(x_1)- Du(x_2)|  \leq  &|Du(x_1)-(Du)_{\Omega_{2\rho}(x_1)}|+|Du(x_2)-(Du)_{\Omega_{\rho}(x_2)}|\\
    & +|D u(x)-(Du)_{\Omega_{2\rho}(x_1)}|+|Du(x)-(Du)_{\Omega_{\rho}(x_2)}|.
\end{align*}
We then take the $L^p$ average over $x\in \Omega_\rho(x_2)\subset\Omega_{2\rho}(x_1)$ and use Lemma \ref{lem:volume}, the Lebesgue differentiation theorem, and the triangle inequality to get
\begin{equation}\label{ddiff1}
\begin{aligned}
    &|Du(x_1)-Du(x_2)|\\
    &\leq  |Du(x_1)-(Du)_{\Omega_{2\rho}(x_1)}|+|Du(x_2)-(Du)_{\Omega_{\rho}(x_2)}|+C\phi(x_1,2\rho)+C\phi(x_2,\rho)\\
    &\leq C \sum_{j=0}^\infty \phi(x_1,2^{1-j}\rho)+C \sum_{j=0}^\infty \phi(x_2,2^{-j} \rho),
\end{aligned}
\end{equation}
where $\phi$ is the mean oscillation of $Du$ defined in \eqref{def:phi}.

By \eqref{delta2}  and the triangle inequality, we know that $\Omega_{x_1,r}\subset\Omega_{x_0,\sqrt{\underline{\delta}(x_0)}/2}\subset\Omega_1$, where $r=\frac{1}{8}\underline{\delta}(x_1)^{1/2}$.
Since \eqref{small-rho} holds, we can apply \eqref{ineq:global1} with $x_1$ in place of $\Bar{x}$ and $2^{1-j}\rho$ in place of $\rho$ to obtain
\begin{align}\label{sum1}
    &\sum_{j=0}^\infty \phi(x_1,2^{1-j} \rho)\leq C\sum_{j=0}^\infty \left(\frac{2^{1-j}\rho}{r}\right)^{\alpha_1}\|D u\|_{L^{\infty}(\Omega_{x_1,r})}\nonumber\\
    &\leq C\left(\frac{\rho}{r}\right)^{\alpha_1}\|D u\|_{L^{\infty}(\Omega_{x_0,\sqrt{\underline{\delta}(x_0)}/2})}  \le C\underline{\delta}(x_0)^{-\alpha_1/2}\rho^{\alpha_1}\|D u\|_{L^\infty(\Omega_{x_0,\sqrt{\underline{\delta}(x_0)}/2})}.
\end{align}
Similarly, we also have
\begin{align}\label{sum2}
    \sum_{j=0}^\infty \phi(x_2,2^{-j} \rho)\leq C\underline{\delta}(x_0)^{-\alpha_1/2}\rho^{\alpha_1}\|D u\|_{L^\infty(\Omega_{x_0,\sqrt{\underline{\delta}(x_0)}/2})}.
\end{align}
Combining \eqref{ddiff1}, \eqref{sum1}, and \eqref{sum2} yields \eqref{c1alpha}.

{\em Case 2: $\rho>\frac{c_4}{36} \underline{\delta}(x_0)$.}
By \eqref{delta2}, we also have
\begin{equation}\label{big-rho}
    \rho> \frac{c_4}{108}\underline{\delta}(x_1)\quad \mbox{and} \quad \rho> \frac{c_4}{108}\underline{\delta}(x_2)
\end{equation}
With $x_1$ in place of $\Bar{x}$ in Section \ref{sec1.4}, we denote the new coordinate by $\xi=\Tilde{\Lambda}_{x_1}(x)$ and set $\xi_1:=(\xi_1',\xi_1^n):=\Tilde{\Lambda}_{x_1}(x_1)$, where $\Tilde{\Lambda}_{x_1}$ is defined as in \eqref{Lambda_T}.
Similarly, with $x_2$ in place of $\Bar{x}$, we denote another coordinate by $\eta:=(\eta_2',\eta_2^n):=\Tilde{\Lambda}_{x_2}(x)$ and set $\eta_2=\Tilde{\Lambda}_{x_2}(x_2)$. Let $u_2^{(1)}$ and $u_2^{(2)}$ be the extended solutions in the coordinates $\xi$ and $\eta$ defined as in Section \ref{sec1.4}, respectively. As in Section \ref{sec1.4}, we also define the mean oscillation of extended solutions in the two coordinates $\xi$ and $\eta$ by
$$
\Tilde{\phi}(x_1,r)= \left(\fint_{B_{r}(\xi_1)}|D_{\xi} u_2^{(1)}-(D_{\xi} u_2^{(1)})_{B_r(\xi_1)}|^{p}\right)^{1/p}
$$
and
$$\Tilde{\phi}(x_2,r) =\left(\fint_{B_{r}(\eta_2)}|D_\eta u_2^{(2)}-(D_\eta u_2^{(2)})_{B_r(\eta_2)}|^{p}\right)^{1/p}.$$
Let us first briefly describe our ideas to prove \eqref{c1alpha} in this case.
By the triangle inequality,
\begin{equation}\label{ddiff2}
\begin{aligned}
    |Du(x_1)-Du(x_2)|
    &\leq  |Du(x_1)-(D_\xi u_2^{(1)})_{B_{c_6\rho}(\xi_1)}|+|Du(x_2)-(D_\eta u_2^{(2)})_{B_{\rho}(\eta_2)}|
    \\&\quad+|(D_\xi u_2^{(1)})_{B_{c_6\rho}(\xi_1)}-(D_\eta u_2^{(2)})_{B_{\rho}(\eta_2)}|,
\end{aligned}
\end{equation}
where $c_6>0$ is a constant which will be specified later. We will estimate the third term on the right hand side of \eqref{ddiff2} using careful change of variables $\xi\to\eta$ and estimate the first two terms using iteration of Lemma \ref{lem:iter}. A delicate transition from the original coordinates $x$ to the new coordinates $\xi$ (or $\eta$) is also needed when the radius (inside the iteration procedure) is at the scale of $\underline{\delta}(x_0)$.

Note that by \eqref{Lambda_T},
$\Phi:=\Tilde{\Lambda}_{x_2}\Tilde{\Lambda}_{x_1}^{-1}$ is indeed a dilation of the coordinate $\xi$ in the $\xi^n$ direction, namely,
$$
\eta=(\eta',\eta^n) = \Phi(\xi)=\big(\xi',\; \frac{h_1(x_2')-h_2(x_2')+\varepsilon}{h_1(x_1')-h_2(x_1')+\varepsilon}\,\xi^n\big).
$$
By \eqref{fg_0},  \eqref{fg_1}, and \eqref{delta2},
we have
\begin{align}\label{Phi-bound}
    \frac{c_1}{9c_2}\leq\frac{c_1|x_2'|^2+\varepsilon}{c_2|x_1'|^2+\varepsilon}\leq \frac{h_1(x_2')-h_2(x_2')+\varepsilon}{h_1(x_1')-h_2(x_1')+\varepsilon}\leq \frac{c_2|x_2'|^2+\varepsilon}{c_1|x_1'|^2+\varepsilon}\leq \frac{9c_2}{c_1}.
\end{align}
This implies that $\Phi$ and $\Phi^{-1}$ are bounded independent of $\varepsilon$. Moreover, similar to \eqref{B-I-comp}, using \eqref{delta2} one can also show that
\begin{equation}\label{coef-1-comp}
    \begin{aligned}
    &\big|\frac{h_1(x_2')-h_2(x_2')+\varepsilon}{h_1(x_1')-h_2(x_1')+\varepsilon}-1\big|\leq 2c_2\frac{(|x_1'|+|x_2'|)|x_1'-x_2'|}{c_1|x_1'|^2+\varepsilon}\leq C\underline{\delta}(x_0)^{-1/2}\rho,\\
    &\big|\frac{h_1(x_1')-h_2(x_1')+\varepsilon}{h_1(x_2')-h_2(x_2')+\varepsilon}-1\big|\leq 2c_2\frac{(|x_1'|+|x_2'|)|x_1'-x_2'|}{c_1|x_2'|^2+\varepsilon}\leq C\underline{\delta}(x_0)^{-1/2}\rho,
\end{aligned}
\end{equation}
where $C>0$ is a constant depending only on $c_1$ and $c_2$. Note that \eqref{coef-1-comp} directly implies
\begin{equation}\label{coef-2-comp}
    \begin{aligned}
    |D \Phi-I_n|\leq C\underline{\delta}(x_0)^{-1/2}\rho \quad \mbox{and} \quad |D \Phi^{-1}-I_n|\leq C\underline{\delta}(x_0)^{-1/2}\rho.
\end{aligned}
\end{equation}
By the definitions of the extended solutions $u_2^{(1)}$ and $u_2^{(2)}$, we also know that for any $\xi,\eta\in \mathbb{R}^n$ with $|\xi'|<1$ and $|\eta'|<1$, we have
$$u_2^{(1)}(\xi)=u_2^{(2)}(\Phi(\xi)), \quad u_2^{(2)}(\eta)=u_2^{(1)}(\Phi^{-1}(\eta)).$$
Without loss of generality, we assume $$\frac{h_1(x_2')-h_2(x_2')+\varepsilon}{h_1(x_1')-h_2(x_1')+\varepsilon}\geq 1, $$
and thus we have $\Phi^{-1}(B_\rho(\eta_2))\subset B_\rho(\xi_2)$, where we denote
 $\xi_2:=(\xi_2',\xi_2^n):=\Phi^{-1}(\eta_2)$. Clearly $\xi_1'=x_1'$ and $\xi_2'=x_2'$.
Therefore, by using \eqref{fg_0}, \eqref{def:c_2}, \eqref{big-rho}, and the triangle inequality, for any $\xi\in \Phi^{-1}(B_\rho(\eta_2))$, it holds that
$$
\begin{aligned}
    |\xi-\xi_1|&\leq |\xi-\xi_2|+|\xi_2-\xi_1|\leq \rho+ |\xi_2'-\xi_1'|+|\xi_2^n-\xi_1^n|\\
    &\leq 2\rho+h_1(x_1')-h_2(x_1')+\varepsilon\\
    &\leq 2\rho+c_2|x_1'|^2+\varepsilon\leq c_6 \rho,
\end{aligned}
$$
where $c_6>2$ is a constant depending only on $n$, $c_1$, and $c_2$.
Thus
    $\Phi^{-1}(B_\rho(\eta_2))\subset B_\rho(\xi_2)\subset B_{c_6\rho}(\xi_1)$.

Let $c_5\in(0,1/8)$ be the same constant as in Lemma \ref{lem:mean3}.
From now on, we assume
\begin{equation}\label{rho-bound}
\rho\leq \min\{c_5/4,1/(64c_6)\} \underline{\delta}(x_0)^{1/2},
\end{equation}
since otherwise \eqref{c1alpha} clearly holds. Note that  \eqref{big-rho}, \eqref{rho-bound}, and \eqref{delta2} directly imply
\begin{equation}\label{c5bound}
     \frac{c_4}{108}\underline{\delta}(x_2)<\rho\leq c_5\underline{\delta}(x_2)^{1/2}<\frac{1}{8}\underline{\delta}(x_2)^{1/2}
\end{equation}
and
\begin{equation}\label{c6bound}
    \frac{c_4}{108}\underline{\delta}(x_1)<c_6\rho\leq \frac{1}{16}\underline{\delta}(x_1)^{1/2}.
\end{equation}
Now we estimate the three terms on the right hand side of \eqref{ddiff2} separately.

We first estimate the term $|(D_\xi u_2^{(1)})_{B_{c_6\rho}(\xi_1)}-(D_\eta u_2^{(2)})_{B_{\rho}(\eta_2)}|$ in \eqref{ddiff2}. For any $\xi\in \Phi^{-1}(B_\rho(\eta_2))$,  by the triangle inequality,
\begin{equation*}
\begin{aligned}
    &|(D_\xi u_2^{(1)})_{B_{c_6\rho}(\xi_1)}-(D_\eta u_2^{(2)})_{B_{\rho}(\eta_2)}|
    \\&\leq |(D_\xi u_2^{(1)})_{B_{c_6\rho}(\xi_1)}-D_\xi u_2^{(1)}(\xi)|+|D_\xi u_2^{(1)}(\xi)-(D_\eta u_2^{(2)})_{B_{\rho}(\eta_2)}|.
\end{aligned}
\end{equation*}
We then take the $L^p$ average over $\xi\in \Phi^{-1}(B_\rho(\eta_2))\subset B_{c_6\rho}(\xi_1)$ and use \eqref{Phi-bound}, \eqref{coef-2-comp}, and the triangle inequality to get
\begin{equation}\label{diff-aver-0}
\begin{aligned}
    &|(D_\xi u_2^{(1)})_{B_{c_6\rho}(\xi_1)}-(D_\eta u_2^{(2)})_{B_{\rho}(\eta_2)}|
    \\&\leq C\Tilde{\phi}(x_1,c_6\rho)+\Big(\fint_{\Phi^{-1}(B_\rho(\eta_2))}|D_\xi u_2^{(1)}(\xi)-(D_\eta u_2^{(2)})_{B_{\rho}(\eta_2)}|^p d\xi\Big)^{1/p}\\
    &\leq C\Tilde{\phi}(x_1,c_6\rho)+\Big(\fint_{B_\rho(\eta_2)}|D_\eta u_2^{(2)}(\eta) \,D\Phi-(D_\eta u_2^{(2)})_{B_{\rho}(\eta_2)}|^p d\eta\Big)^{1/p}\\
    &\leq C\Tilde{\phi}(x_1,c_6\rho)+C\Tilde{\phi}(x_2,\rho)+C \underline{\delta}(x_0)^{-1/2}\rho\,\|D_\eta u_2^{(2)}\|_{L^\infty(B_\rho(\eta_2))}.
\end{aligned}
\end{equation}
By \eqref{c5bound} and \eqref{co-diff4}
we know that
$|D\Tilde{\Lambda}_{x_2}|\leq C$ in $\Omega_{x_2,\rho}$ and thus
\begin{equation}\label{gradient-compare}
    \|D_\eta u_2^{(2)}\|_{L^\infty(B_\rho(\eta_2))}\leq C\|Du\|_{L^\infty(\Omega_{x_2,\rho})}\leq C\|Du\|_{L^\infty(\Omega_{x_0,\sqrt{\underline{\delta}(x_0)}/2})}.
\end{equation}
By \eqref{c6bound}, we can apply
\eqref{ineq:global2}  with $x_1$ in place of $\Bar{x}$ and $c_6\rho$ in place of $\rho$,  and use \eqref{delta2} to get
\begin{equation}\label{1-phi-t}
\begin{aligned}
    &\Tilde{\phi}(x_1,c_6\rho)\leq C\underline{\delta}(x_0)^{-\alpha_1/2}\rho^{\alpha_1}\|Du\|_{L^\infty(\Omega_{x_0,\sqrt{\underline{\delta}(x_0)}/4})}\\
     &\leq C\underline{\delta}(x_0)^{-\alpha_1/2}\rho^{\alpha_1}\|Du\|_{L^\infty(\Omega_{x_0,\sqrt{\underline{\delta}(x_0)}/2})}.
\end{aligned}
\end{equation}
Similarly, we have
\begin{equation}\label{2-phi-t}
\begin{aligned}
    &\Tilde{\phi}(x_2,\rho)\leq C\underline{\delta}(x_0)^{-\alpha_1/2}\rho^{\alpha_1}\|Du\|_{L^\infty(\Omega_{x_2,\sqrt{\underline{\delta}(x_0)}/4})}
    \\&\leq C\underline{\delta}(x_0)^{-\alpha_1/2}\rho^{\alpha_1}\|Du\|_{L^\infty(\Omega_{x_0,\sqrt{\underline{\delta}(x_0)}/2})}.
\end{aligned}
\end{equation}
Combining \eqref{diff-aver-0}, \eqref{gradient-compare}, \eqref{1-phi-t} and \eqref{2-phi-t}, we obtain
\begin{equation}\label{average-diff}
    |(D_\xi u_2^{(1)})_{B_{c_6\rho}(\xi_1)}-(D_\eta u_2^{(2)})_{B_{\rho}(\eta_2)}|
    \leq C\underline{\delta}(x_0)^{-\alpha_1/2}\rho^{\alpha_1}
    \|Du\|_{L^\infty(\Omega_{x_0,\sqrt{\underline{\delta}(x_0)}/2})}.
\end{equation}
Next, we estimate the term $|Du(x_1)-(D_\xi u_2^{(1)})_{B_{c_6\rho}(\xi_1)}|$ in \eqref{ddiff2}.
We define $\rho_j:=c_62^{-j}\rho$ for $j\in\mathbb{N}$. Because of \eqref{big-rho}, we let $j_1\geq 1$ be the integer such that
\begin{align}\label{def-j1}
    \rho_{j_1}\geq \frac{c_4}{216} \underline{\delta}(x_1),\quad \rho_{j_1+1}<\frac{c_4}{216}\underline{\delta}(x_1).
\end{align}
Then by using the triangle inequality and Lemma \ref{lem:volume}, we have
\begin{equation}\label{1-diff-aver}
    \begin{aligned}
        &|Du(x_1)-(D_\xi u_2^{(1)})_{B_{c_6\rho}(\xi_1)}|
        \\&\leq |Du(x_1)-(Du)_{\Omega_{\rho_{j_1+1}}(x_1)}|+|(Du)_{\Omega_{\rho_{j_1+1}}(x_1)}-(D_\xi u_2^{(1)})_{B_{\rho_{j_1}}(\xi_1)}|
        \\&\quad+|(D_\xi u_2^{(1)})_{B_{\rho_{j_1}}(\xi_1)}-(D_\xi u_2^{(1)})_{B_{ c_6\rho}(\xi_1)}|\\
        &\leq C\sum_{j=j_1+1}^\infty \phi(x_1,\rho_j)+C\sum_{j=0}^{j_1}\Tilde{\phi}(x_1,\rho_j)+|(Du)_{\Omega_{\rho_{j_1+1}}(x_1)}-(D_\xi u_2^{(1)})_{B_{\rho_{j_1}}(\xi_1)}|.
    \end{aligned}
\end{equation}
Using Lemma \ref{lem:volume} and H\"older's inequality, we obtain
\begin{equation}\label{diff-aver-j}
\begin{aligned}
    &|(Du)_{\Omega_{\rho_{j_1+1}}(x_1)}-(D_\xi u_2^{(1)})_{B_{\rho_{j_1}}(\xi_1)}|
   \\
    &\leq C\Big(\fint_{\Omega_{\rho_{j_1}}(x_1)}|D u(x)-(D_\xi u_2^{(1)})_{B_{\rho_{j_1}}(\xi_1)}|^p dx\Big)^{1/p}\\
    &\leq C\Big(\fint_{\Tilde{\Lambda}_{x_1}(\Omega_{\rho_{j_1}}(x_1))}|D_\xi u_2^{(1)}(\xi)\,B(\xi)-(D_\xi u_2^{(1)})_{B_{\rho_{j_1}}(\xi_1)}|^p d\xi\Big)^{1/p},
\end{aligned}
\end{equation}
where as in Section \ref{sec1.4}, we denote
$$B(\xi)=D\Tilde{\Lambda}_{x_1}(\Tilde{\Lambda}_{x_1}^{-1}(\xi)),$$
and use \eqref{co-diff4} in the last line.
By \eqref{co-diff3} with $x_1$ in place of $\Bar{x}$, from \eqref{diff-aver-j} we deduce
\begin{equation}\label{diff-aver2}
\begin{aligned}
    &|(Du)_{\Omega_{\rho_{j_1+1}}(x_1)}-(D_\xi u_2^{(1)})_{B_{\rho_{j_1}}(\xi_1)}|
       \leq C\Tilde{\phi}(x_1,2\rho_{j_1})+\frac{C\rho_{j_1}}{\underline{\delta}(x_1)^{1/2}} \|Du\|_{L^\infty(\Omega_{x_1,\rho_{j_1}})}
       \\&\leq C\Tilde{\phi}(x_1,\rho_{j_1-1})+C\rho^{1/2}\|Du\|_{L^\infty(\Omega_{x_1,c_6\rho})}.
\end{aligned}
\end{equation}
Here in the last inequality we also used \eqref{def-j1} and \eqref{big-rho}.
Combining \eqref{1-diff-aver}, \eqref{diff-aver2} and using \eqref{rho-bound}, we get
\begin{equation}\label{diff-aver3}
    \begin{aligned}
         &|Du(x_1)-(D_\xi u_2^{(1)})_{B_{c_6\rho}(\xi_1)}|
        \\
        &\leq C\sum_{j=j_1+1}^\infty \phi(x_1,\rho_j)+C\sum_{j=0}^{j_1}\Tilde{\phi}(x_1,\rho_j)+C\rho^{1/2}\|Du\|_{L^\infty(\Omega_{x_1,\sqrt{\underline{\delta}(x_0)}/4})}.
    \end{aligned}
\end{equation}
We recall \eqref{big-rho} and \eqref{c6bound}. Therefore, by applying Lemma \ref{lem:iter}  with $x_1$ in place of $\Bar{x}$ and $\rho_j$ in place of $\rho$, and using  \eqref{diff-aver3} and \eqref{delta2}, we obtain
\begin{equation}\label{diff-aver4}
\begin{aligned}
    &|Du(x_1)-(D_\xi u_2^{(1)})_{B_{c_6\rho}(\xi_2)}|
        \\& \leq C\underline{\delta}(x_0)^{-\alpha_1/2} \rho^{\alpha_1} \|Du\|_{L^\infty(\Omega_{x_1,\sqrt{\underline{\delta}(x_0)}/4})}+C\rho^{1/2}\|Du\|_{L^\infty(\Omega_{x_1,\sqrt{\underline{\delta}(x_0)}/4})}
        \\&\leq
        C\underline{\delta}(x_0)^{-\alpha_1/2} \rho^{\alpha_1} \|Du\|_{L^\infty(\Omega_{x_0,\sqrt{\underline{\delta}(x_0)}/2})}.
\end{aligned}
\end{equation}
Similarly, it also holds that
\begin{equation}\label{diff-aver5}
    |Du(x_2)-(D_\xi u_2^{(2)})_{B_{\rho}(\xi_2)}|
        \leq
        C\underline{\delta}(x_0)^{-\alpha_1/2} \rho^{\alpha_1} \|Du\|_{L^\infty(\Omega_{x_0,\sqrt{\underline{\delta}(x_0)}/2})}.
\end{equation}
Combining \eqref{ddiff2}, \eqref{average-diff}, \eqref{diff-aver4}, and \eqref{diff-aver5} yields \eqref{c1alpha}. The proof is completed.
\end{proof}

\section{Estimates of \texorpdfstring{$U_1^\varepsilon-U_2^\varepsilon$}{|U1-U2|}}\label{sec:U1-U2}

In the following, we use the $C^{1,\beta}$ estimate we derive in Proposition \ref{thm-1/2} to obtain an asymptotic expansion of $Du_\varepsilon$ in terms of $U_1^\varepsilon - U_2^\varepsilon$, for arbitrary constants $U_1^\varepsilon$ and $U_2^\varepsilon$ (Proposition \ref{prop_gradient_lowerbound_U1-U2}). When $p \ge (n+1)/2$,  for the specific $U_1^\varepsilon$ and $U_2^\varepsilon$ in \eqref{equinfty}, $U_1^\varepsilon-U_2^\varepsilon$ will converges to $0$ as $\varepsilon \to 0$, and we compute in Theorem \ref{thm_U1-U2} the rate of convergence using information of the flux $\cF$  defined in \eqref{flux}. When $p < (n+1)/2$, $U_1^\varepsilon-U_2^\varepsilon$ will converges to $U_1 - U_2$ as shown in Theorem \ref{thm_convergence}.

\begin{proposition}\label{prop_gradient_lowerbound_U1-U2}
Let $h_1$, $h_2$ be  $C^{2}$ functions satisfying \eqref{fg_0}-\eqref{def:c_2}, $p>1$, $n \ge 2$, $\varepsilon \in[0,1/4)$, $U_1^\varepsilon$, $U_2^\varepsilon$ be arbitrary constants with $ |U_i^\varepsilon|  \le  \|\varphi\|_{L^\infty(\partial\Omega)}$, and $u_\varepsilon \in W^{1,p}(\widetilde\Omega^\varepsilon)$ be a solution of
$$
\left\{
\begin{aligned}
-\dv (|D u_\varepsilon|^{p-2} D u_\varepsilon) &=0  &&\mbox{in }\widetilde{\Omega}^\varepsilon,\\
u_\varepsilon &= U_i^\varepsilon &&\mbox{on}~\partial{\cD_{i}^\varepsilon},~i=1,2,\\
 u_\varepsilon &= \varphi  &&\mbox{on } \partial \Omega.
\end{aligned}
\right.
$$
There exist positive constants $\beta \in (0, 1)$, $C_1$ and $C_2$ depending only on $n$, $p$, $c_1$, and $c_2$, such that
\begin{equation}\label{asymptotes_U1-U2}
Du_\varepsilon(x)=\left(0',\frac{U_1^\varepsilon - U_2^\varepsilon}{\delta(x)}\right)+\mathbf{f}_1(x,\varepsilon) \quad \mbox{for}~~x \in \overline\Omega_{1/4}^\varepsilon,
\end{equation}
where $\delta$ is defined in \eqref{def_delta},
\begin{equation}\label{f1_bound_U1-U2}
|\mathbf{f}_1(x,\varepsilon)|\le C_1 \left( \frac{|U_1^\varepsilon - U_2^\varepsilon|}{\delta^{1-\beta/2}(x)} + \|\varphi\|_{L^\infty(\partial\Omega)}e^{-\frac{C_2}{\sqrt\varepsilon + |x'|}} \right).
\end{equation}

\end{proposition}

\begin{proof}
By mean value theorem, we know that for any $x=(x',x_n)\in \overline\Omega_{1/4}^\varepsilon$, there exist $\zeta(x')\in(-\frac{\varepsilon}{2}+h_2(x'),\frac{\varepsilon}{2}+h_1(x'))$ and $y(x)=(x',\zeta(x'))\in \Omega_{1/4}^\varepsilon$, such that
\begin{equation*}
    D_n u_\varepsilon(y(x))=(U_1^{\varepsilon}-U_2^{\varepsilon})/\delta(x).
\end{equation*}
Let \begin{equation}\label{f1-def}
    \mathbf{f}_1(x,\varepsilon):=(D_{x'} u_\varepsilon(x), D_n u_\varepsilon(x)-D_n u_\varepsilon(y(x))).
\end{equation}
By Propositions \ref{thm-1/2} and \ref{prop_gradient_upperbound_U1-U2}, we have
\begin{equation}\label{f_1-bound}
\begin{aligned}
    | D_n u_\varepsilon(x)-D_n u_\varepsilon(y(x))|&\le C |x_n-\zeta(x')|^{\beta}{\delta}(x)^{-\beta/2}\|D u_\varepsilon\|_{L^\infty(\Omega_{x,\sqrt{\underline{\delta}(x)}/2})}
    \\&\le C_1\left( \frac{|U_1^\varepsilon - U_2^\varepsilon|}{\delta^{1-\beta/2}(x)} + \|\varphi\|_{L^\infty(\partial\Omega)}e^{-\frac{C_2}{\sqrt\varepsilon + |x'|}} \right).
\end{aligned}
\end{equation}
Let $z(x)=(x', \varepsilon/2 +h_1(x'))$. Since $u_\varepsilon\equiv U_1^\varepsilon$ on $\overline{\cD}_1^\varepsilon$, by Proposition \ref{prop_gradient_upperbound_U1-U2}, we have
\begin{equation}\label{f2-bound}
\begin{aligned}
    |D_{x'}u(z(x))|&\leq C |x'| |Du(z(x))|\\
    &\leq C_1 |x'|\left( \frac{|U_1^\varepsilon - U_2^\varepsilon|}{\delta(x)} + \|\varphi\|_{L^\infty(\partial\Omega)}e^{-\frac{C_2}{\sqrt\varepsilon + |x'|}} \right).
\end{aligned}
\end{equation}
Similar to \eqref{f_1-bound}, we also have
\begin{equation}\label{f3-bound}
    \begin{aligned}
    |D_{x'}u(x)-D_{x'}u(z(x))|\le C_1 \left( \frac{|U_1^\varepsilon - U_2^\varepsilon|}{\delta^{1-\beta/2}(x)} + \|\varphi\|_{L^\infty(\partial\Omega)}e^{-\frac{C_2}{\sqrt\varepsilon + |x'|}} \right).
    \end{aligned}
\end{equation}
Therefore, by \eqref{f2-bound}, \eqref{f3-bound}, and the triangle inequality, we have
\begin{equation*}
    |D_{x'}u(x)|\leq C_1 \left( \frac{|U_1^\varepsilon - U_2^\varepsilon|}{\delta^{1-\beta/2}(x)} + \|\varphi\|_{L^\infty(\partial\Omega)}e^{-\frac{C_2}{\sqrt\varepsilon + |x'|}} \right).
\end{equation*}
This completes the proof of Proposition \ref{prop_gradient_lowerbound_U1-U2}.
\end{proof}

Following the proof of \cite{GorNov}*{Proposition 2.1} with slight modification, 
we have the following proposition, which will be proved in the Appendix to make this article self-contained.

\begin{proposition}\label{prop_flux}
Let $n \ge 2$, $p \ge (n + 1)/2$, $h_1$, $h_2$ be  $C^2$ functions satisfying \eqref{fg_0}-\eqref{def:c_2}, $\varepsilon \in(0,1)$, and $u_\varepsilon \in W^{1,p}(\Omega)$ be the solution of \eqref{equinfty}. Then there exist positive constants $C_1, C_2$ depending only on $n$, $p$, $c_1$, $c_2$, and $\|\varphi\|_{L^\infty}$, such that for any $r \in (0,1)$,
\begin{equation}\label{flux_close_to_F}
\left| \cF - \lim_{\varepsilon \to 0_+}\int_{\Gamma_{-,r}^\varepsilon} |D u_\varepsilon|^{p-2} D u_\varepsilon \cdot \nu \right| \le C_1 e^{-\frac{C_2}{r}},
\end{equation}
where $\cF$ is given in \eqref{flux}.
\end{proposition}

\begin{remark}
The proof of Proposition \ref{prop_flux} relies on the facts that $|D u_0| \le C_1 e^{-\frac{C_2}{r}}$ and $\| u_\varepsilon - u_0 \|_{C^{1,\alpha}(K)} \to 0$ as $\varepsilon \to 0$ for any $K  \subset \subset \Omega \setminus \Big( \underset{0 < \varepsilon \le \varepsilon_0}{\cup}(\cD_1^\varepsilon \cup \cD_2^\varepsilon) \cup \{0\} \Big)$ with $\varepsilon_0>0$, where $u_0$ is the minimizer of \eqref{minimizer} with $\varepsilon = 0$. See Proposition \ref{prop_gradient_upperbound_U1-U2} and Theorem \ref{thm_convergence}. However, $D u_0$ may be unbounded when $p < (n + 1)/2$. This fact was overlooked in \cites{CirSci,CirSci2}.
\end{remark}

With the help of Propositions \ref{prop_gradient_upperbound_U1-U2} and \ref{prop_flux}, we are able to derive the rate of convergence for $U_1^\varepsilon - U_2^\varepsilon$ when $\varepsilon \to 0$.

\begin{theorem}\label{thm_U1-U2}
Let $p \ge (n + 1)/2$, $U_1^\varepsilon$ and $U_2^\varepsilon$ be the constants in \eqref{equinfty}. Then it holds that
\begin{equation}\label{U1-U2}
 \lim_{\varepsilon \to 0_+} (U_1^\varepsilon - U_2^\varepsilon) \Theta(\varepsilon)^{-1} = \sgn(\cF)(K|\cF|)^{1/(p-1)},
\end{equation}
where $\Theta(\varepsilon)$ is given in \eqref{Theta}, $K$ is given in \eqref{def:K}, and $\cF$ is given in \eqref{flux}.
\end{theorem}

\begin{proof}
By Proposition \ref{prop_gradient_lowerbound_U1-U2}, we have
\begin{equation}\label{gradient_lowerbound_U1-U2}
\left|Du_\varepsilon \cdot \nu (y) - \frac{U_1^\varepsilon - U_2^\varepsilon}{\delta(y)}\right| \le C_1 \Big( \frac{|U_1^\varepsilon - U_2^\varepsilon|}{\delta^{1-\beta/2}(y)} + e^{-\frac{C_2}{\sqrt\varepsilon + |y'|}} \Big) \quad\mbox{for}~~y \in \Gamma_{-,1/4}^\varepsilon,
\end{equation}
where $\delta$ is defined in \eqref{def_delta}.

We will show that every sequence converges to the same limit.

First, Let $\{\varepsilon_j\}_{j\in\mathbb{N}}$ be a decreasing sequence such that $\varepsilon_j \to 0$ as $j\to \infty$ and $U_1^{\varepsilon_j}\ge U_2^{\varepsilon_j}$ for every $j\in\mathbb{N}$.
By \eqref{gradient_lowerbound_U1-U2}, for any $\tau\in(0,1/2)$ and $r \in (0, 1/4)$, we have
\begin{align*}
& (1-\tau) \Big(  \frac{U_1^{\varepsilon_j} - U_2^{\varepsilon_j}}{\delta_j(y)} \Big)^{p-1}(1 - C \delta_j^{\beta/2}(y)) - C(\tau) e^{-\frac{C_2}{\sqrt{\varepsilon_j} + |y'|}} \le |D u_{\varepsilon_j}|^{p-2} D u_{\varepsilon_j} \cdot \nu (y) \\
& \le  (1+\tau) \Big(  \frac{U_1^{\varepsilon_j} - U_2^{\varepsilon_j}}{\delta_j(y)} \Big)^{p-1}(1 + C \delta_j^{\beta/2}(y)) + C(\tau) e^{-\frac{C_2}{\sqrt{\varepsilon_j} + |y'|}} \quad \text{for}\,\, y \in \Gamma_{-,1/4}^{\varepsilon_j},
\end{align*}
where $\delta_j(y):=\varepsilon_j+h_1(y')-h_2(y')$, $C, C_2 > 0$ depends only on $n$, $p$, $c_1$, $c_2$, $\beta$, $\|\varphi\|_{L^\infty}$, and $\dist(\cD_1^{\varepsilon_j} \cup \cD_2^{\varepsilon_j}, \partial\Omega)$, and $C(\tau)$ additionally depends on $\tau$. By the change of variables $dS = \sqrt{1 + |D_{x'} h_2(y')|^2} dy'$, \eqref{fg_0}, and \eqref{def:c_2}, we have for any $r < 1/4$,
\begin{align}\label{U_1-U_2_flux}
&(1-\tau)\int_{|y'| < r} \Big( \frac{U_1^{\varepsilon_j} - U_2^{\varepsilon_j}}{\delta_j(y)}\Big)^{p-1} (1 - C \delta_j^{\beta/2}(y)) \sqrt{1 + |D_{x'} h_2(y')|^2} \, dy' \nonumber\\
&\quad -C(\tau) \int_{|y'| < r} e^{-\frac{C_2}{\sqrt{\varepsilon_j} + |y'|}}\, dy'\nonumber\\
&\le   \int_{\Gamma_{-,r}^{\varepsilon_j}} |D u_{\varepsilon_j}|^{p-2} D u_{\varepsilon_j} \cdot \nu \,dS\nonumber\\
& \le  (1+\tau)\int_{|y'| < r} \Big( \frac{U_1^{\varepsilon_j} - U_2^{\varepsilon_j}}{\delta_j(y)}\Big)^{p-1} (1 + C \delta_j^{\beta/2}(y)) \sqrt{1 + |D_{x'} h_2(y')|^2} \, dy' \nonumber\\
&\quad +C(\tau) \int_{|y'| < r} e^{-\frac{C_2}{\sqrt{\varepsilon_j} + |y'|}}\, dy'.
\end{align}
Note that
\begin{align}\label{error_term}
\int_{|y'| < r} e^{-\frac{C_2}{\sqrt{\varepsilon_j} + |y'|}}\, dy' \le  Cr^{(n-1)}.
\end{align}
Moreover, for $p \ge (n+1)/2$,
\begin{equation}\label{calculus}
\lim_{r \to 0_+} \lim_{{\varepsilon} \to 0_+} \int_{|y'| < r} \Big( \frac{\Theta({\varepsilon})}{\delta(y)}\Big)^{p-1} \, dy'= \frac{1}{K}.
\end{equation}
The verification of \eqref{calculus} follows from direct computations, which will be given in the Appendix (Lemma \ref{lemma_calculus}).
By taking the limit as ${j} \to \infty$ in \eqref{U_1-U_2_flux} first, then taking $r \to 0$ and $\tau \to 0$, from \eqref{flux_close_to_F}, \eqref{U_1-U_2_flux}, \eqref{error_term}, and \eqref{calculus} we get
\begin{equation}\label{limit-pos}
\lim_{j\to \infty} \big((U_1^{\varepsilon_j} - U_2^{\varepsilon_j})\Theta(\varepsilon_j)^{-1}\big)^{p-1} = K\cF.
\end{equation}
Similarly,  if $\{\varepsilon_j\}_{j\in\mathbb{N}}$ is a decreasing sequence such that $\varepsilon_j \to 0$ as $j\to \infty$ and $U_1^{\varepsilon_j}\le U_2^{\varepsilon_j}$ for every $j\in\mathbb{N}$, then we have
\begin{equation}\label{limit-neg}
\lim_{j\to \infty} \big((U_2^{\varepsilon_j} - U_1^{\varepsilon_j})\Theta(\varepsilon_j)^{-1}\big)^{p-1} = -K\cF.
\end{equation}
Therefore, we conclude that if $\cF>0$, then for any decreasing sequence $\varepsilon_j\to 0_+$, there exists $j_0\in\mathbb{N}$ such that $U_1^{\varepsilon_j}\ge U_2^{\varepsilon_j}$ for every $j\geq j_0$, since otherwise there exists a decreasing subsequence $\{\varepsilon_{j_k}\}$ such that $U_1^{\varepsilon_{j_k}}< U_2^{\varepsilon_{j_k}}$, then from \eqref{limit-neg} we should have $\cF\le 0$. Thus if $\cF>0$, \eqref{limit-pos} implies \eqref{U1-U2}. Similarly, \eqref{U1-U2} also holds when $\cF<0$. For the remaining case when $\cF=0$, let $\{\varepsilon_j\}$ be any decreasing sequence such that $\varepsilon_j\to 0$ as $j\to \infty$. Then there exists a decreasing subsequence $\{\varepsilon_{j_k}\}$ such that either $U_1^{\varepsilon_{j_k}}\le U_2^{\varepsilon_{j_k}}$ holds for every $k\in \mathbb{N}$ or $U_1^{\varepsilon_{j_k}}\ge U_2^{\varepsilon_{j_k}}$ holds for every $k\in \mathbb{N}$. Thus \eqref{limit-pos} (or \eqref{limit-neg}) implies that
$$
\lim_{k\to \infty} (U_1^{\varepsilon_{j_k}} - U_2^{\varepsilon_{j_k}})\Theta(\varepsilon_{j_k})^{-1} = 0.
$$
Therefore, \eqref{U1-U2} is also true when $\cF=0$. The proof of the theorem is completed.
\end{proof}

\section{Proofs of Theorem \ref{thm:expansion} and Proposition \ref{thm_example}}\label{sec:proofs}
In this section, we give the proofs of Theorem \ref{thm:expansion} and Proposition \ref{thm_example}.

\begin{proof}[Proof of Theorem \ref{thm:expansion}]
First, we consider the case when $p\ge (n+1)/2$.
Let \begin{equation}\label{U_1-U_2-limit-def}
    f_0(\varepsilon):=(U_1^{\varepsilon}-U_2^{\varepsilon})/\Theta(\varepsilon)-\sgn(\cF)(K|\cF|)^{1/(p-1)}.
    \end{equation}
By Theorem \ref{thm_U1-U2} and \eqref{f1_bound_U1-U2}, we know that
\begin{equation*}
    \lim_{\varepsilon\to0}f_0(\varepsilon)=0,
\end{equation*}
and
$$
|\mathbf{f}_1(x,\varepsilon)|\le C_1 \Big( {\delta}(x)^{\beta/2-1}\Theta(\varepsilon)(|K\cF|^{1/(p-1)}+|f_0(\varepsilon)|)+ \|\varphi\|_{L^\infty(\partial\Omega)}e^{-\frac{C_2}{\sqrt\varepsilon + |x'|}}\Big),
$$
where $\mathbf{f}_1$ is defined as in \eqref{f1-def}. Then \eqref{eq:thm-f1} follows from \eqref{asymptotes_U1-U2}, \eqref{U_1-U_2-limit-def}, and the above.

Next, we consider the case when $1<p<(n+1)/2$. We define
\begin{equation*}
    g_0(\varepsilon):=U_1^{\varepsilon}-U_2^{\varepsilon}-(U_1-U_2).
\end{equation*}
and $\mathbf{g}_1(x,\varepsilon) = \mathbf{f}_1(x,\varepsilon)$ as in \eqref{f1-def}. By Theorem \ref{thm_convergence},
\begin{equation*}
    \lim_{\varepsilon\to0}g_0(\varepsilon)=0,
\end{equation*}
Similar as above, \eqref{eq:thm-g1} follows.
\end{proof}

\begin{proof}[Proof of Proposition \ref{thm_example}]
For the case when $p \ge (n + 1)/2$, by the symmetry of the domain and uniqueness of the solution, we know that $u_0(x', x_n) = -u_0(x', - x_n)$. Therefore, $U_0 = u(x',0) = 0$. By the strong maximum principle, $u_0 > 0$ in $\widetilde\Omega^0 \cap \{x_n > 0\}$. Then by the Hopf lemma, $Du_0 \cdot \nu > 0$ on $\partial \cD_1^0$, and hence $\cF > 0$.

For the case when $1<p < (n + 1)/2$, we prove by contradiction. Assume $U_1 \le U_2$. Since $u_0(x', x_n) = -u_0(x', - x_n)$, we know that $u_0(x',0) = 0$ and $U_1 = - U_2$. Therefore, $U_1 \le 0$ and $u_0$ achieves minimum in $B_5^+$ on $\partial \cD_1^0$. Hence $Du_0 \cdot \nu > 0$ on $\partial \cD_1^0$ by the Hopf lemma. This implies
$$
\int_{\partial \cD_1^0} |D u_0|^{p-2} D u_0 \cdot \nu > 0,
$$
which contradicts to \eqref{touching_equation_2}$_3$.
\end{proof}

\appendix

\section{}

In the first part of the appendix, we prove Proposition \ref{prop_flux}. The proof essentially follows those of \cite{GorNov}*{Proposition 2.1} and \cite{CirSci}*{Lemma 5.1}. Our estimate is sharper due to a better estimate on $|Du_0|$ (Proposition \ref{prop_gradient_upperbound_U1-U2}).

\begin{proof}[Proof of Proposition \ref{prop_flux}]
Similar to the proof of Theorem \ref{thm_convergence}, for small $r \in (0,1/2)$, we take a smooth surface $\eta$ so that $\cD_1^\varepsilon$ is surrounded by $\Gamma_{-,s}^0 \cup \eta$. See Figure \ref{zeroflux}. We denote the surface
$$
\Sigma_r^\varepsilon := \left\{ x \in \bR^n: |x'|=r, -\frac{\varepsilon}{2} + h_2(x') < x_n < h_2(x') \right\}.
$$
Since $\int_{\partial \cD_1^{0}} |D u_{0}|^{p-2} D u_{0} \cdot \nu = \cF$ and $\int_{\partial \cD_1^{\varepsilon}} |D u_{\varepsilon}|^{p-2} D u_{\varepsilon} \cdot \nu =0$, by integration by parts, we have
\begin{equation}\label{appendix1}
-\int_{\Gamma_{-,r}^0} |D u_{0}|^{p-2} D u_{0} \cdot \nu + \int_{\eta} |D u_{0}|^{p-2} D u_{0} \cdot \nu = \cF,
\end{equation}
and
\begin{equation}\label{appendix2}
-\int_{\Gamma_{-,r}^\varepsilon} |D u_{\varepsilon}|^{p-2} D u_{\varepsilon} \cdot \nu + \int_{\Sigma_r^\varepsilon} |D u_{\varepsilon}|^{p-2} D u_{\varepsilon} \cdot \nu + \int_{\eta} |D u_{\varepsilon}|^{p-2} D u_{\varepsilon} \cdot \nu =0.
\end{equation}
Note that the minus signs appear because $\nu$ on $\Gamma_{-,r}^0$ and $\Gamma_{-,r}^\varepsilon$ are defined to be pointing upwards, while $\nu$ on $\eta$ and $\Sigma_r^\varepsilon$ are pointing away from $\cD_1^\varepsilon$. By \eqref{gradient_upperbound_U1-U2}, we have $|D u_0(x)| \le C_1 e^{-\frac{C_2}{r}}$ in $\overline{\Omega_{r}^0}$, and hence
\begin{equation}\label{appendix3}
\left|\int_{\Gamma_{-,r}^0} |D u_{0}|^{p-2} D u_{0} \cdot \nu \right| \le C_1 e^{-\frac{C_2}{r}}
\end{equation}
for some positive $\varepsilon$-independent constants $C_1$ and $C_2$. By Theorem \ref{thm_convergence}, we have
\begin{equation}\label{appendix4}
\lim_{\varepsilon \to 0_+} \int_{\eta} |D u_{\varepsilon}|^{p-2} D u_{\varepsilon} \cdot \nu = \int_{\eta} |D u_{0}|^{p-2} D u_{0} \cdot \nu.
\end{equation}
By \eqref{gradient_upperbound_U1-U2-new}, we have $|D u_\varepsilon(x)| \le C (\varepsilon + |x'|^2)^{-1} $ in ${\Omega_{1/2}^0}$. Therefore,
\begin{align}\label{appendix5}
\left| \int_{\Sigma_r^\varepsilon} |D u_{\varepsilon}|^{p-2} D u_{\varepsilon} \cdot \nu \right| &\le C \int_{|x'|=r}\int_{-\varepsilon/2+h_2(x')}^{h_2(x')} \Big( \frac{1}{\varepsilon+|x'|^2} \Big)^{p-1} \, dx_n dS \nonumber\\
&\le C\frac{\varepsilon r^{n-2}}{(\varepsilon+r^2)^{p-1}}.
\end{align}
Finally, \eqref{flux_close_to_F} follows directly from \eqref{appendix1}-\eqref{appendix5}. Proposition \ref{prop_flux} is proved.
\end{proof}

In the following, we verify \eqref{calculus}.
\begin{lemma}\label{lemma_calculus}
\eqref{calculus} holds when $p \ge (n+1)/2$.
\end{lemma}
\begin{proof}
We only give the proof for the case when $n \ge 3$. The case $n = 2$ follows similarly and is simpler. After a rotation of coordinates if necessary, we may assume that
$$
D_{x'}^2 (h_1 - h_2)(0') = \mbox{diag}~(\lambda_1, \ldots , \lambda_{n-1}).
$$
First, we replace $\delta(y)$ in the denominator with the quadratic polynomial $\varepsilon + \sum_{i=1}^{n-1} \lambda_i y_i^2/2$. By \eqref{fg_1}, \eqref{def:c_2}, and the fact that $h_1,h_2$ are $C^2$, we estimate
\begin{align*}
&\left| \int_{|y'| < r} \Big( \frac{\Theta(\varepsilon)}{\delta(y)}\Big)^{p-1} - \Big( \frac{\Theta(\varepsilon)}{\varepsilon + \sum_{i=1}^{n-1} \frac{\lambda_i}{2} y_i^2}\Big)^{p-1}  \, dy' \right|\\
&= \Theta(\varepsilon)^{p-1 } \left| \int_{|y'| < r} \frac{\Big(\varepsilon + \sum_{i=1}^{n-1} \frac{\lambda_i}{2} y_i^2\Big)^{p-1} - \delta(y)^{p-1} }{\Big[ \delta(y) \Big(\varepsilon + \sum_{i=1}^{n-1} \frac{\lambda_i}{2} y_i^2\Big) \Big]^{p-1}} \, dy' \right|\\
&\le C \Theta(\varepsilon)^{p-1} \int_{|y'| < r}  \frac{h(r)|y'|^2(\varepsilon + |y'|^2)^{p-2}}{(\varepsilon + |y'|)^{2p-2}} \, dy'\\
&\le C h(r) \int_{|y'| < r} \Big( \frac{\Theta(\varepsilon)}{\varepsilon + \sum_{i=1}^{n-1} \frac{\lambda_i}{2} y_i^2}\Big)^{p-1}  \, dy',
\end{align*}
where $h(r)$ is the modulus of continuity of $D_{x'}^2 (h_1 - h_2)$, and hence $h(r) \to 0$ as $r \to 0$, and $C$ is some positive constant independent of $\varepsilon$ and $r$. Therefore, it suffices to show that for any $r > 0$,
\begin{equation}\label{calculus_goal}
\lim_{\varepsilon \to 0_+} \int_{|y'| < r} \Big( \frac{\Theta(\varepsilon)}{\varepsilon + \sum_{i=1}^{n-1} \frac{\lambda_i}{2} y_i^2}\Big)^{p-1}  \, dy' = \frac{1}{K}.
\end{equation}
In the spherical coordinates, for $y' \in \bR^{n-1}$, we write
$$
y_1=\sqrt{\frac{2}{\lambda_1}}s \cos\theta_1,\quad y_2= \sqrt{\frac{2}{\lambda_2}}s\sin\theta_1\cos\theta_2,\quad
y_3=\sqrt{\frac{2}{\lambda_3}}s\sin\theta_1\sin\theta_2\cos\theta_3,\ldots,
$$
$$
y_{n-2}=\sqrt{\frac{2}{\lambda_{n-2}}}s \sin\theta_1\sin\theta_2\cdots\sin\theta_{n-3}\cos\theta_{n-2},
$$
$$
y_{n-1}=\sqrt{\frac{2}{\lambda_{n-1}}}s\sin\theta_1\sin\theta_2\cdots\sin\theta_{n-3}\sin\theta_{n-2},
$$
where $s \in [0, \infty)$, $\theta_1,\theta_2,\ldots,\theta_{n-3}\in [0,\pi]$ and $\theta_{n-2}\in [0,2\pi)$. For convenience of notation, we denote $\Sigma:=[0,\pi]^{n-3}\times [0,2\pi)$. By this change of variables,
\begin{align}\label{spherical_change}
&\int_{|y'| < r} \Big( \frac{\Theta(\varepsilon)}{\varepsilon + \sum_{i=1}^{n-1} \frac{\lambda_i}{2} y_i^2}\Big)^{p-1}  \, dy' \nonumber \\
&= \frac{2^{\frac{n-1}{2}}}{(\lambda_1 \cdots \lambda_{n-1})^{\frac{1}{2}}} \int_\Sigma \int_0^{\frac{r}{\varphi(\theta)}}\Big( \frac{\Theta(\varepsilon)}{\varepsilon + s^2}\Big)^{p-1} s^{n-2} J(\theta)\,dsd\theta,
\end{align}
where
$$
\varphi(\theta) = \Big( \frac{2}{\lambda_1}\cos^2\theta_1 + \frac{2}{\lambda_2} \sin^2\theta_1\cos^2\theta_2 + \cdots + \frac{2}{\lambda_{n-1}}\sin^2\theta_1\cdots\sin^2\theta_{n-2} \Big)^{\frac{1}{2}},
$$
and
$$
J(\theta)=\sin^{n-3}\theta_1 \sin^{n-4}\theta_2\cdots \sin\theta_{n-3}.
$$
Note that
$$
\sqrt{\frac{2}{\max \lambda_i}} \le \varphi(\theta) \le \sqrt{\frac{2}{\min \lambda_i}}.
$$

When $p > (n+1)/2$, $\Theta(\varepsilon) = \varepsilon^{\frac{2p-n-1}{2(p-1)}}$. By the change of variables $t = \varepsilon^{-1}s^2$, the right-hand side of \eqref{spherical_change} becomes
\begin{align*}
\frac{2^{\frac{n-3}{2}}}{(\lambda_1 \cdots \lambda_{n-1})^{\frac{1}{2}}} \int_\Sigma \int_0^{\frac{r^2}{\varphi^2(\theta)\varepsilon}} \frac{t^{\frac{n-3}{2}}}{(1+t)^{p-1}} \, dt J(\theta) d\theta.
\end{align*}
Since $(n-3)/2 - (p-1) = (n-2p-1)/2 < -1$, the integral converges as $\varepsilon \to 0$. Therefore,
\begin{align}\label{calculus_case1}
\lim_{\varepsilon \to 0_+} \int_{|y'| < r} \Big( \frac{\Theta(\varepsilon)}{\varepsilon + \sum_{i=1}^{n-1} \frac{\lambda_i}{2} y_i^2}\Big)^{p-1}  \, dy' &= \frac{2^{\frac{n-3}{2}}}{(\lambda_1 \cdots \lambda_{n-1})^{\frac{1}{2}}} \int_\Sigma \int_0^{\infty} \frac{t^{\frac{n-3}{2}}}{(1+t)^{p-1}} \, dt J(\theta) d\theta \nonumber \\
&= \frac{2^{\frac{n-3}{2}}|\bS^{n-2}|}{(\lambda_1 \cdots \lambda_{n-1})^{\frac{1}{2}}} B\Big(\frac{n-1}{2}, \frac{2p-n-1}{2} \Big),
\end{align}
where $B$ is the beta function. Recalling the identities
\begin{equation}\label{gamma_function}
|\bS^{n-2}| = \frac{2\pi^{\frac{n-1}{2}}}{\Gamma\big(\frac{n-1}{2} \big)}, \quad B\Big(\frac{n-1}{2}, \frac{2p-n-1}{2} \Big) = \frac{\Gamma\big(\frac{n-1}{2} \big)\Gamma\big(p -\frac{n-1}{2} \big)}{\Gamma(p-1)},
\end{equation}
and plugging them into \eqref{calculus_case1}, we have proved \eqref{calculus_goal} for the case when $p > (n+1)/2$.

When $p = (n+1)/2$, $\Theta(\varepsilon) = |\ln \varepsilon|^{-\frac{1}{p-1}}$. By the change of variables $w = \varepsilon^{-1/2}s$, the right-hand side of \eqref{spherical_change} becomes
\begin{align*}
&\frac{2^{\frac{n-1}{2}}|\ln \varepsilon|^{-1}}{(\lambda_1 \cdots \lambda_{n-1})^{\frac{1}{2}}} \int_\Sigma \int_0^{\frac{r}{\varphi(\theta)\sqrt\varepsilon}}\frac{w^{n-2}}{(1+w^2)^{\frac{n-1}{2}}} \, dw J(\theta)\, d\theta\\
&= \frac{2^{\frac{n-1}{2}}|\ln \varepsilon|^{-1}}{(\lambda_1 \cdots \lambda_{n-1})^{\frac{1}{2}}} \int_\Sigma \int_0^{\frac{r}{\varphi(\theta)\sqrt\varepsilon}}\frac{w}{1+w^2} \, dw J(\theta)\, d\theta\\
&+ \frac{2^{\frac{n-1}{2}}|\ln \varepsilon|^{-1}}{(\lambda_1 \cdots \lambda_{n-1})^{\frac{1}{2}}} \int_\Sigma \int_0^{\frac{r}{\varphi(\theta)\sqrt\varepsilon}} \left( \frac{w^{n-2}}{(1+w^2)^{\frac{n-1}{2}}} - \frac{w}{1+w^2} \right) \, dw J(\theta)\, d\theta\\
&=: {\rm I} + {\rm II}.
\end{align*}
By direct computations,
\begin{align*}
{\rm I} &= \frac{2^{\frac{n-3}{2}}|\ln \varepsilon|^{-1}}{(\lambda_1 \cdots \lambda_{n-1})^{\frac{1}{2}}} \int_\Sigma \int_0^{\frac{r}{\varphi(\theta)\sqrt\varepsilon}} \, d \ln (1 + w^2) J(\theta)\, d\theta\\
&= \frac{2^{\frac{n-3}{2}}|\ln \varepsilon|^{-1}}{(\lambda_1 \cdots \lambda_{n-1})^{\frac{1}{2}}} \int_\Sigma \Big[\ln \big( \varepsilon + \frac{r^2}{\varphi^2(\theta)} \big) - \ln \varepsilon \Big] J(\theta)\, d\theta.
\end{align*}
Therefore, by \eqref{gamma_function},
\begin{equation}\label{calculus_case2}
\lim_{\varepsilon \to 0_+} {\rm I} = \frac{2^{\frac{n-3}{2}}|\bS^{n-2}|}{(\lambda_1 \cdots \lambda_{n-1})^{\frac{1}{2}}} = \frac{(2\pi)^{\frac{n-1}{2}}}{(\lambda_1 \cdots \lambda_{n-1})^{\frac{1}{2}}\Gamma\big(\frac{n-1}{2} \big)}.
\end{equation}
To estimate ${\rm II}$, we split the integral over $(0, \frac{r}{\varphi(\theta)\sqrt\varepsilon})$ into $(0,1)$ and $(1, \frac{r}{\varphi(\theta)\sqrt\varepsilon})$, and denote them by ${\rm II}_1$ and ${\rm II}_2$, respectively. It is easily seen that $|{\rm II}_1| \le C|\ln \varepsilon|^{-1}$. To estimate ${\rm II}_2$, we have
$$
|{\rm II}_2| \le C|\ln \varepsilon|^{-1} \int_\Sigma \int_1^{\frac{r}{\varphi(\theta)\sqrt\varepsilon}} \frac{w \Big[ w^{n-3} - (1+w^2)^{\frac{n-3}{2}} \Big]}{(1 + w^2)^{\frac{n-1}{2}}}\, dw J(\theta)\, d\theta.
$$
By the mean value theorem, there exists a $\xi \in (w^2, 1+w^2)$, such that
$$
w^{n-3} - (1+w^2)^{\frac{n-3}{2}} = - \frac{n-3}{2} \xi^{\frac{n-5}{2}}.
$$
Note that $(w^2, 1 + w^2) \subset (w^2 , 2w^2)$ when $w \ge 1$. Therefore,
$$
\left| \int_1^{\frac{r}{\varphi(\theta)\sqrt\varepsilon}} \frac{w \Big[ w^{n-3} - (1+w^2)^{\frac{n-3}{2}} \Big]}{(1 + w^2)^{\frac{n-1}{2}}}\, dw \right| \le C\int_1^\infty \frac{w^{n-4}}{(1 + w^2)^{\frac{n-1}{2}}} \le C,
$$
which implies $|{\rm II}_2| \le C|\ln \varepsilon|^{-1}$. Hence, by \eqref{calculus_case2} and the estimate $|{\rm II}_1| + |{\rm II}_2| \le C|\ln \varepsilon|^{-1}$, we have proved \eqref{calculus_goal} for the case when $p = (n+1)/2$.
\end{proof}

\bibliographystyle{amsplain}

\end{document}